\documentclass[12pt]{article}
\usepackage{amsfonts}
\usepackage{amsmath,amsthm}

\newcommand{\Zb}{\mathbb{Z}}

\newcommand{\LL}{\mathcal{L}}
\newcommand{\Hess}{{\rm Hess}}
\newcommand{\Ric}{{\rm Ric}}
\theoremstyle{definition}
\newtheorem{remark}{Remark}
\newcommand{\vol}{{\rm vol}}

\newcommand{\comment}[1]{} 

\newcommand{\R}{\mathbb{R}}
\newcommand{\Z}{\mathbb{Z}}

\newcommand{\Gammatt}{\widetilde{\Gamma}_2}
\newcommand{\asum}[1]{\widetilde{\sum_{#1}}}
\numberwithin{equation}{section}
\newtheorem{prop}[equation]{Proposition}

\newtheorem{lemma}[equation]{Lemma}

\newtheorem{definition}[equation]{Definition}

\newtheorem{corollary}[equation]{Corollary}

\newtheorem{theorem}[equation]{Theorem}

\begin{document}
\title{Li-Yau inequality on graphs}
\author{Frank Bauer, Paul Horn, Yong Lin,\\ Gabor Lippner,  Dan Mangoubi, Shing-Tung Yau}
\maketitle

\abstract{
We prove the Li-Yau gradient estimate for the heat kernel on graphs. The only assumption is a variant of the curvature-dimension inequality, which is purely local, and can be considered as a new notion of curvature for graphs. We compute this curvature  
for lattices and trees and conclude that it behaves more naturally than the already existing notions of curvature.
Moreover, we show that if a graph has non-negative curvature
then it has polynomial volume growth.

We also derive Harnack inequalities and heat kernel bounds from the gradient estimate, and show how it can be used to strengthen the classical Buser inequality relating the spectral gap and the Cheeger constant of a graph.}

\section{Introduction and main ideas}
%\subsection{Background and main achievements}
In their celebrated work \cite{LiYau} Li and Yau  proved an upper bound on the gradient of positive solutions of the heat equation. In its simplest form, for an $n$-dimensional compact manifold with non-negative Ricci curvature the Li-Yau gradient estimate states that a positive solution $u$ of the heat equation \mbox{$(\Delta - \partial_t)u = 0$} satisfies
\begin{equation}
|\nabla\log u|^2-\partial_t(\log u)=\frac{|\nabla u|^2}{u^2} - \frac{\partial_t u}{u} \leq \frac{n}{2t}. \label{eq:ly}
\end{equation}
The  inequality (\ref{eq:ly}) has been generalized to many important settings in geometric analysis. The most notable one was made by Hamilton on the Ricci flow, see \cite{Ham1,Ham2}.

Finding a discrete version of~\eqref{eq:ly} has proven challenging for a long time. Indeed, inequality~\eqref{eq:ly} is not true even on the lattice $\Zb^n$. The main difficulty for finding a discrete version is that the chain rule \emph{fails} on graphs.
In this paper, we succeed in finding an analogue of inequality~\eqref{eq:ly} on graphs.  The main breakthrough and novelty of this paper as we see it is twofold. First, we show a way to bypass the chain rule in the discrete setting. The way we do it (as explained in \S\ref{sec:chain-rule-idea}), we believe,
may be adapted to many other circumstances.
 Second, we introduce a new natural notion of curvature of graphs, modifying the curvature notion
of~\cite{LinYau}. For example, we are able to prove a discrete version of the Li-Yau inequality where the curvature is bounded from below (by any real number). Also,  we show that
non-negatively curved graphs have polynomial growth.
As far as we are aware of, this result, well known on
Riemannian manifolds, is not known  with any previous notion of curvature on graphs. 

In the next two sections we explain the preceding two ideas in more detail.

%For a long time, there was difficulty to give a discrete version of Li yau inequality .

%Of course, this is very natural in light of the corresponding %result on manifolds, but 
%we are not aware of any such kind of result on graphs
% with the already known curvature notions on graphs.
% This gives an indication that the new notion of curvature
% we introduce  may be natural to work with.
% 
%We compute the curvature in several examples. In particular we show that trees can have negative curvature $K$ with $|K|$ arbitrarily large. So far the existing notions of curvature~(\cite{LinYau}, \cite{Oll}) always gave $K\geq -2$ for trees. This is another indication that the new notion 
%of curvature is natural.

% Besides proving an analogue of \eqref{eq:ly} for non-negatively %curved graphs, we also recover generalizations to graphs with %curvature bounded from below as well as analogues to the case of %non-compact manifolds (corresponding to infinite graphs).  We %establish analogues for several discrete Laplace operators %including the normalized and ordinary graph Laplacian.  We also %establish gradient estimates for solutions to Schr\"odinger's %equation $(\Delta - \partial_t - q)u = 0$.
%===============================
%\subsection{Two novel key ideas}
%\label{sec:ideas}
%===============================

\subsection{Bypassing the chain rule - discretizing the logarithm}
\label{sec:chain-rule-idea}

In proving the gradient estimate~\eqref{eq:ly} on manifolds,
either by the maximum principle~\cite{LiYau} or by semigroup methods~\cite{BakLed} it is crucial to have the chain rule in hand. Namely, both proofs use a simple but a key identity that follows from the chain rule formula:
\begin{equation}\label{eq:logid}
\Delta \log u = 
%\frac{\Delta u}{u} - \frac{|\nabla u|^2}{u^2} = %
 \frac{\Delta u}{u} - |\nabla \log u|^2.
\end{equation}
However, this is false in the discrete setting. Even worse,  there seems to be no way to reasonably bound the difference of the two sides. 

The lack of the chain rule on graphs is the main difficulty in trying to prove a discrete analogue
of~\eqref{eq:ly}.
 In what follows we explain our solution to this issue.
  First, we find a one parameter family of simple identities on manifolds which resembles~\eqref{eq:logid}:  For every 
  $p > 0$ one has
\begin{equation}\label{eq:pid} \Delta u^p =
%  p u^{p-1}\Delta u + p(p-1)u^{p-2}|\nabla u|^2 = %
 pu^{p-1}\Delta u + \frac{p-1}{p} u^{-p}|\nabla u^p|^2.\end{equation}
These also follow from the chain rule. Then, we make the following
crucial observation:
% The original identity (\ref{eq:logid}) can be considered as the %limit of this family
%as $p\to 0$. 
While there exists no chain rule in the discrete setting, 
quite remarkably, identity~\eqref{eq:pid} for $p=1/2$ still holds  on graphs. This fact is the starting point and probably the most important observation of  this paper.

Naively, this means that each time identity~\eqref{eq:logid} is applied in the proof of~\eqref{eq:ly} we may try to replace it by identity~\eqref{eq:pid} with $p=1/2$.
Indeed, this idea starts our work in this paper. 
However, this idea alone is not enough to prove a discrete analogue of~\eqref{eq:ly}: We have to redefine the notion of curvature on graphs as explained in the next section.
%====================================================
\subsection{A new notion of curvature for graphs}
\label{sec:new-curvature-idea}
%====================================================
The second obstacle we have to overcome in proving gradient estimates on graphs is that a proper notion of curvature on graphs is not a priori clear. It is a well known problem to extend the notion of Ricci curvature, or more precisely to define lower bounds for the Ricci curvature in more general spaces than Riemannian manifolds. At present a lot of research has been done in this direction (see e.g. \cite{DK, LinYau, LV,  Oht, Oll, Stu1}).  The approach to generalizing curvature in the context of gradient estimates by the use of curvature-dimension inequalities explained below was pioneered by Bakry and Emery \cite{BakEm}. 

On a Riemannian manifold $M$ Bochner's identity reveals a connection between harmonic functions or, more generally, 
solutions of the heat equation and the Ricci curvature.
It is given by
$$\forall f\in C^{\infty}(M)\quad\frac{1}{2}\Delta |\nabla f|^2 = \langle \nabla f, \nabla \Delta f \rangle + \|\Hess f\|_2^2 +  \Ric(\nabla f, \nabla f).$$
 An immediate consequence of the Bochner identity is that on an $n$-dimensional manifold whose Ricci curvature is bounded from below by $K$ one has
\begin{equation}\label{eq:cd}
\frac{1}{2}\Delta |\nabla f|^2 \geq \langle \nabla f, \nabla \Delta f \rangle + \frac{1}{n}(\Delta f)^2 + K |\nabla f|^2,
\end{equation} 
which is called the \emph{curvature-dimension inequality} (CD-inequality).  It was an important insight by Bakry and Emery \cite{BakEm} that one can use it as a substitute for the lower Ricci curvature bound on spaces where a direct generalization of Ricci curvature is not available. 

Since all known proofs of the Li-Yau gradient estimate exploit non-negative curvature condition through the  CD-inequality~\eqref{eq:cd}, one would believe it is a natural choice in our case as well. Bakry and Ledoux \cite{BakLed} succeed to use it to generalize~\eqref{eq:ly} to Markov operators on general measure spaces when the operator satisfies a chain rule type formula.

 As we have explained in Section~\ref{sec:chain-rule-idea} there is no chain rule in 
the discrete setting.
However, due to formula~\eqref{eq:pid} with $p=1/2$ which  compensates
for the lack of the chain rule, we succeed to modify  the standard CD-inequality on graphs in order to define a \emph{new} curvature notion on graphs (cf.\ \S\ref{sec:cd}) which we can use to prove
a discrete gradient estimate in Theorem~\ref{thm:compactgre}.

One may argue that as we modify the curvature notion
it might not be natural anymore. In fact, we show it is natural
in several respects: First, we prove
that our modified CD-inequality follows from 
the classical one in situations where the chain rule does hold (Theorem~\ref{thm:cd->cde}).
Second, we compute it in several examples (\S\ref{sec:examples}) to show it gives
reasonable results. In particular, 
 we show that trees can have negative curvature $K$ with $|K|$ arbitrarily large. So far the existing notions of curvature~\cite{LinYau, Oll} always gave $K\geq -2$ for trees.
 Third, as  mentioned above, we derive polynomial volume growth for graphs satisfying
 non-negative curvature condition, 
 like on manifolds (Corollary~\ref{cor:polynomial-growth}),
 and it seems to be a first result of this kind on graphs.
 %We are not aware of any such kind of results on graphs
 %with the already existing curvature notions on graphs.

\subsection{Background on the parabolic Harnack 
inequality on graphs}

%In many important problems in geometry, we only know the local %behavior. Thus, inequalities like the gradient estimate are very %valuable.

Inequality~\eqref{eq:ly} can be integrated over space-time, and some new distance function on space-time can be introduced to measure the ratio of the positive solution at different points:
\begin{equation} \label{eq:harnack}
u(x,s) \leq C(x,y,s,t) u(y,t)\ ,
\end{equation} where $C(x,y,s,t)$ depends only on the distance of $(x,s)$ and $(y,t)$ in space-time. Using this, \cite{LiYau} also gave a sharp estimate of the heat kernel in terms of such a distance function.
 
The Harnack inequality~\eqref{eq:harnack} has many applications. Besides implying bounds on the heat kernel, it can be used to prove eigenvalue estimates, and it is one of the main techniques in the regularity theory of PDEs. Hence it is important to decide what manifolds satisfy such an inequality.  Grigor'yan~\cite{Gri} and Saloff-Coste~\cite{SC} gave a complete characterization of such manifolds. They showed that satisfying a volume doubling property  along with a Poincar\'e inequality is actually equivalent to satisfying the Harnack inequality. The characterization by Grigor'yan and Saloff-Coste generalizes the non-negative Ricci curvature condition by Li and Yau, since it is known from the work of Buser \cite{Buser} that a lower bound on the Ricci curvature implies volume doubling and the  Poincar\'e inequality.  However, a major drawback of the characterization by Grigr'yan and Saloff-Coste  is that showing that a manifold satisfies these properties is rather difficult as both volume doubling and the Poincar\'e inequality are global in nature. The results in~\cite{LiYau}  have the advantage that a simple local condition, a lower bound on curvature, is sufficient to guarantee that the more global properties hold.

In the case of graphs Delmotte~\cite{Delmotte} proved a characterization analogous to that of Grigor'yan and Saloff-Coste. However, just as for manifolds, his conditions are hard to verify because of their global nature. One virtue of our results is that they give \emph{local} conditions that imply Harnack type inequalities.
 \vspace{2ex}

%In this paper we also show how to derive the Harnack %inequality~\eqref{eq:harnack} in the discrete
%setting from our gradient estimate, and prove heat kernel bounds %for graphs that have curvature bounded from below. 
%Finally, we derive polynomial volume growth for graphs
%satisfying a non-negative curvature condition.

\textbf{Organization of the paper.}  In Section~\ref{sec:setup} we set up the scope of this paper.
In Section~\ref{sec:cd} we define our notion of curvature by modifying the standard curvature-dimension inequality, and we study the basic properties of the curvature. In particular, we show that
on manifolds the modified CD-inequality follows from the classical one. 
Our main results are contained in Section \ref{gradest}, where we establish the discrete analogue of the gradient estimate~\eqref{eq:ly}.  In Section \ref{harnack}, we use the gradient estimates to derive Harnack inequalities.
Section~\ref{sec:examples} contains curvature computations for certain classes of graphs. In particular we give a general lower bound for graphs with bounded degree and show that this bound is asymptotically sharp in the case of trees.  We also show that lattices, and more generally Ricci-flat graphs in the sense of Chung and
Yau \cite{ChY}, have non-negative curvature. 
Finally, in Section \ref{app} we apply our results to derive heat kernel bounds, polynomial volume growth
 and prove a Buser-type eigenvalue estimate.

\section{Setup and notations}\label{sec:setup}

First we fix our notation. Let $G=(V,E)$ be a graph.  We allow the edges on the graph to be weighted; that is, the edge $xy$ from $x$ to $y$ has weight $w_{xy} > 0$.  We do not require that the edge weights be symmetric, so $w_{xy} \neq w_{yx}$ in general, for the proofs of the main theorems, but our key examples satisfying the curvature condition do have symmetric weights.  We do, however, require that
\[
\inf_{e \in E} w_e =: w_{min} > 0.  
\]  
Moreover we assume in the following that the graph is locally finite, i.e. $\deg(x):=\sum_{y\sim x}w_{xy}<\infty$ for all $x\in V$. 

Given a finite measure $\mu: V \to \mathbb{R}$ on $V$, the $\mu$-Laplacian on $G$ is  the operator $\Delta: \mathbb{R}^{|V|} \to \mathbb{R}^{|V|}$ 
%To describe the heat equation we use the Laplacian $\Delta: \mathbb{R}^{|V|} \to \mathbb{R}^{|V|}$
 defined by
\[
\Delta f (x) = \frac{1}{\mu(x)} \sum_{y \sim x} w_{xy}(f(y) - f(x)).
\]
Since such averages will appear numerous times in computations, we introduce an abbreviated notation for ``averaged sum'':  For a vertex $x \in V$,
\[ \asum{y \sim x} h(y) := \frac{1}{\mu(x)} \sum_{y \sim x} w_{xy} h(y).
\] 

Given a graph and measure, we define
\[
D_w = \max_{\substack{x,y \in V\\x \sim y}}\frac{\deg(x)}{w_{xy}}
\]
and
\[
D_\mu = \max_{x \in V} \frac{\deg(x)}{\mu(x)}.
\]

So far as is possible, we will treat $\mu$-Laplacian operators generally.  The
special cases of most interest, however, are the cases where $\mu \equiv 1$ which is
the standard graph Laplacian, and the case where $\mu(x) = \sum_{y \sim x} w_{xy}=\deg(x)$, 
which yields the normalized graph Laplacian.  In the case where the edges are 
unweighted, we take $w_{xy} \equiv 1$.  Throughout the remainder of the paper, we will
simply refer to the $\mu$-Laplacian as the Laplacian, except when it is important
to emphasize the effect of the measure.

In this paper, we are interested in functions $u : V \times [0,\infty) \to \R$ that are solutions of the heat equation.  Let us introduce the operator 
\[ \LL = \Delta - \partial_t. \]
We say that $u(x,t)$ is a positive solution to the heat equation, if $u > 0$ and $\LL u =0$. It is not hard to see that such solutions can be written as $u(x,t) = P_tu_0$ where $P_t = e^{t\Delta}$ is the heat kernel and
$u_0 = u(\cdot, 0)$.   Note that the heat equation of course also depends on the measure $\mu$, through the Laplacian it contains.  

\section{Curvature-dimension inequalities} \label{sec:cd}

In this section we introduce a new version of the CD-inequality, which is one of the key steps in deriving analogues of the Li-Yau gradient estimate. We also compare our new notion to the standard CD-inequality. First we need to recall \cite{BakLed} the definition of two natural bilinear forms associated to the Laplacian.

\begin{definition}
The gradient form $\Gamma$ is defined by
\begin{align*}
2\Gamma(f,g)(x) & = \big(\Delta(f \cdot g) - f \cdot \Delta(g) - \Delta(f) \cdot g\big)(x) = \\ &= \frac{1}{\mu(x)} \sum_{y \sim x} w_{xy} (f(y) - f(x))(g(y)-g(x)).
\end{align*}
We write $\Gamma(f) = \Gamma(f,f)$.
\end{definition}

Similarly,
\begin{definition}
The iterated gradient form is defined by
\[ 
2\Gamma_2(f,g) = \Delta \Gamma(f,g)  - \Gamma(f, \Delta g) - \Gamma(\Delta f, g),
\]
We write $\Gamma_2(f) = \Gamma_2(f,f)$.
\end{definition}

\begin{definition}
We say that a graph $G$ satisfies the CD-inequality  $CD(n,K)$ if, for any function $f$

\[
\Gamma_2 (f) \geq \frac{1}{n} (\Delta f)^2 + K\Gamma(f). \] Note that this is exactly the CD-inequality in \eqref{eq:cd} written in the $\Gamma$ notation. 
$G$ satisfies $CD(\infty,K)$ if
\[
\Gamma_2(f) \geq K\Gamma(f).  
\]

\end{definition}
%As we have pointed out in the introduction, in the case of an %$n$-dimensional manifold whose curvature is bounded from below by %$K$ this inequality is a simple consequence of Bochner's identity. %However, following the work of Bakry and Emery \cite{BakEm}, it %has proven
%to be an important {\it definition} of curvature in many other %settings. 

The main example in the definition below is the
Laplace-Beltrami operator on a  Riemannian manifold:
\begin{definition} The semigroup $P_t  = e^{t \Delta} $ is said to be a \textit{diffusion semigroup} if the following identities are satisfied for any smooth function $\Phi$: 
\begin{align}
\Gamma(f, g \cdot h) & = g \cdot \Gamma(f,h)+ h \cdot \Gamma(f,g) \\
\Gamma(\Phi\circ f, g) & = \Phi'(f) \Gamma(f,g) \label{eq:chvar1} \\ 
\Delta(\Phi\circ f) &= \Phi'(f) \Delta(f) + \Phi''(f) \Gamma(f). \label{eq:chvar2}
\end{align}

Bakry and Ledoux \cite{BakLed} show that if the operator $\Delta$ satisfying $CD(n,0)$ generates a diffusion semigroup then the gradient estimate~\eqref{eq:ly} holds.  
\end{definition}

The Laplacian $\Delta$ we are interested in does \emph{not} generate a diffusion semigroup, but remarkably, as we mentioned in the introduction,  for the choice of $\Phi(f)= \sqrt{f}$ a key formula similar to a combination of \eqref{eq:chvar1} and \eqref{eq:chvar2} still holds:
\begin{equation}\label{eq:keyformula}
 2 \sqrt{u} \Delta \sqrt{u} =\Delta u  - 2\Gamma(\sqrt{u}).
\end{equation}
 
 This motivates the following key modification of the CD-inequality.

\begin{definition}
\label{def:cde}
We say that a graph $G$ satisfies the {\it exponential curvature dimension inequality} at the point $x\in V$, $CDE(x, n, K)$ if  for  any positive function $f : V\to \R$ such that $(\Delta f)(x) <0$ we have
$$\Gamma_2(f)(x)  - \Gamma\left(f, \frac{\Gamma(f)}{f}\right)(x) \geq \frac{1}{n} (\Delta f)(x)^2 + K \Gamma(f)(x)\ .$$
We say that $CDE(n, k)$ is satisfied if $CDE(x, n, K)$ is satisfied for all $x\in V$.
%
%or equivalently, using~\eqref{eq:keyformula}
%$$\frac{1}{2}\Delta \Gamma(f)  - 
%\Gamma\left(f, \frac{\Delta(f^2)}{2f}\right)\geq \frac{1}{n} %(\Delta f)^2 + K  \Gamma(f)\ .$$
\end{definition}
%We say that $G$ satisfies $CDE(\infty,K)$ if
%\[
% \Gamma_2(f)  - \Gamma\left(f, \frac{\Gamma(f)}{f}\right)  \geq %K\Gamma(f)\ .  
%\]

\begin{remark}
For convenience, we set
\begin{equation} \Gammatt(f) := \Gamma_2(f) - \Gamma\left(f, \frac{\Gamma(f)}{f}\right) \ .
\end{equation}
By~\eqref{eq:keyformula}
\begin{equation}
\label{eq:Gamma2-equivalence}
\Gammatt(f)= \frac{1}{2}\Delta \Gamma(f)  - \Gamma\left(f, \frac{\Delta(f^2)}{2f}\right)\ .
\end{equation}
\end{remark}

\begin{remark}
An important aspect of both $CD(n,k)$ and $CDE(n,k)$ is that they are local properties.  That is, satisfying $CD(n,k)$ or $CDE(n,k)$ at a point depends only
on the second neighborhood of a vertex.  Thus, in principle, it is possible to classify all (unweighted) graphs
which satisfy $CDE(n,k)$ and have maximum degree at most $D$.

Of course, one hopes that typical graphs which one might consider to have non-negative curvature satisfy $CDE(n,0)$ for some ``dimension" $n$ .  As we will show in Section~\ref{sec:examples}, the class of Ricci-flat graphs~\cite{ChY}, which includes abelian Cayley graphs and most
notably the lattices $\mathbb{Z}^d$ (along with finite tori) do indeed satisfy $CDE(2d,0)$.  
\end{remark}

\begin{remark}
 The reason we chose the adjective ``exponential" in Definition~\ref{def:cde} is revealed in Lemma~\ref{lem:gamma-gamma2} below. 
\end{remark}

\begin{lemma}
\label{lem:gamma-gamma2}
 If the semigroup generated by $\Delta$ is a diffusion semigroup,
 then for any positive function $f$ one has
 $$\Gammatt(f)=f^2\Gamma_2(\log f)\ .$$
\end{lemma}

\begin{proof}  We compute
\begin{align*}
2\Gamma_2(\log f) =&  \Delta \Gamma(\log f) - 2\Gamma(\log f, \Delta \log f) \\=&\Delta\left( \frac{\Gamma(f)}{f^2}\right) - \frac{2}{f}\Gamma\left(f, \frac{\Delta f}{f} - \frac{\Gamma(f)}{f^2}\right) \\ =&
 \frac{\Delta \Gamma(f)}{f^2} + 2\Gamma \left(\frac{1}{f^2}, \Gamma(f)\right) + \Gamma(f) \Delta \left(\frac{1}{f^2}\right) - \frac{2\Gamma(f,\Delta f)}{f^2} - \frac{2 \Delta f}{f}\Gamma(f, f^{-1}) \\ &+ \frac{2}{f^3}\Gamma(f, \Gamma(f)) + \frac{2 \Gamma(f)}{f} \Gamma(f, f^{-2})\\ = 
 &\frac{2\Gamma_2(f)}{f^2} - \frac{4}{f^3}\Gamma(f,\Gamma(f)) - 2\Gamma(f) \frac{\Delta(f)}{f^3} + 6 \frac{\Gamma(f)^2}{f^4} + \frac{2}{f^3}\Delta f \Gamma(f)\\& + \frac{2}{f^3} \Gamma(f,\Gamma(f)) - 4\frac{\Gamma(f)^2}{f^4}  = \frac{2\Gamma_2(f)}{f^2} - \frac{2}{f^3}\Gamma(f,\Gamma(f)) + 2 \frac{\Gamma(f)^2}{f^4} \\=& \frac{2\Gamma_2(f)}{f^2} - \frac{2}{f^3}\Gamma(f,\Gamma(f)) - 2 \frac{\Gamma(f)}{f^2}\Gamma(f, f^{-1}) \\ =& \frac{2}{f^2}\left( \Gamma_2(f)- \Gamma\left(f, \frac{\Gamma(f)}{f}\right) \right)=
 \frac{2\Gammatt(f)}{f^2}\ .
\end{align*}
\end{proof}

\begin{theorem}
\label{thm:cd->cde}
 If the semigroup generated by $\Delta$ is a diffusion semi\-group,  then the condition $CD(n,K)$ implies $CDE(n,K)$. \label{prop:imply}
\end{theorem}
\begin{proof}
Let $f$ be a positive function such that $(\Delta f)(x)<0$.
By Lemma~\ref{lem:gamma-gamma2}
\begin{equation}
\label{ineq:lemma-step}
\Gammatt(f) = f^2 \Gamma_2(\log f) \geq f^2\left(\frac{1}{n}(\Delta \log f)^2 + K\cdot \Gamma(\log f)\right)
=\frac{1}{n}f^2(\Delta\log f)^2+K\Gamma(f)\ .
\end{equation}
On the other hand,
\begin{equation}
\label{ineq:f-delta-log-f}
f(x)\Delta(\log f)(x) = (\Delta f)(x) - \frac{\Gamma(f)(x)}{f(x)} \leq (\Delta f)(x)<0\ ,
\end{equation}
Squaring~\eqref{ineq:f-delta-log-f} and inserting the result in~\eqref{ineq:lemma-step}
yields
$$\Gammatt(f)(x)\geq \frac{1}{n}(\Delta f)(x)^2+K\Gamma(f)(x)\ .$$
\end{proof}

\begin{remark}
In light of Lemma~\ref{lem:gamma-gamma2} it is tempting to define a graph to satisfy  the condition $CDE'(n,K)$ if for all $f > 0$,
\[\Gammatt(f) \geq \frac{1}{n} f^2(\Delta \log f)^2 + K\Gamma(f),\]
and use this (which implies $CDE(n,K)$) instead of $CDE$.  Indeed, in the case of diffusion semigroups $CD(n,K)$ and $CDE'(n,K)$ are equivalent.

Rather interestingly, making such a definition in the graph case loses something:  First, as we show below in Theorem \ref{thm:ricciflat}, the integer grid $\Z^d$ satisfies $CDE(2d,0)$.  On the other hand, it only satisfies $CDE'(4.53d,0)$ and this dimension constant essentially cannot be improved.  Second, it turns out that some graphs (and, in particular, regular trees) do not satisfy $CDE'(n,-K)$ for any $K > 0$.  In contrast, we show in Theorem \ref{thm:lb} below that all graphs satisfy $CDE(2,-K)$ for some $K > 0$. 
\end{remark}

%==============================================
\section{Gradient estimates} \label{gradest}
%===============================================
In this section we prove discrete analogues of the Li-Yau gradient estimate~\eqref{eq:ly} for graphs satisfying the CDE-inequality.
\subsection{Preliminaries}
 The following lemma, describing the behavior of a function near its local maximum, will be used repeatedly throughout the whole section. 

\begin{lemma}\label{lemma:L(gF)} Let $G(V,E)$ be a (finite or infinite) graph, and let $g ,F : V \times [0,T] \to \R$ be functions. Suppose that $g(x,t) \geq 0$, and $F(x,t)$ has a local maximum at $(x^*,t^*) \in V \times [0,T]$.  Further assume $t^* \neq 0$. Then
%\begin{align*}
%\Delta (gF)(x^*,t^*) &\leq (\Delta g)F(x^*,t^*), \\
%\partial_t (gF)(x^*,t^*) &\geq (\partial_t g)F(x^*,t^*),
%\end{align*}
%and in particular 
\[ \LL(gF)(x^*,t^*) \leq (\LL g)F(x^*,t^*).\]
\end{lemma}
\begin{proof}

\begin{align*}
\Delta (gF)(x^*,t^*) &= \frac{1}{\mu(x^*)} \sum_{y \sim x^*} w_{x^*y}(g(y,t^*) F(y,t^*) - g(x^*,t^*) F(x^*,t^*)) \\
&\leq \frac{1}{\mu(x^*)} \sum_{y \sim x^{*}}  w_{x^*y}(g(y,t^*)F(x^*,t^*) -  g(x^*,t^*)F(x^*,t^*)) \\
&= (\Delta g)F(x^*,t^*).
\end{align*}
Similarly 
\begin{align*}
\partial_t(gF)(x^*,t^*) = (\partial_t g)F(x^*,t^*) + g(\partial_t F)(x^*,t^*) \geq (\partial_t g)F(x^*,t^*),
\end{align*}
since $\partial_t F = 0$ at the local maximum if $0 < t^* < T$ and $\partial_t F \geq 0$ if $t^* = T$. The last claim is just the difference of the previous two.
\end{proof}

For convenience, we also record here some simple facts which we use repeatedly in 
our proofs of the gradient estimates.  

\begin{lemma} \label{prop:easy}
Suppose $f: V \to \mathbb{R}$ satisfies $f> 0$, and $(\Delta f)(x) < 0$ at
some vertex $x$.  
Then
\begin{alignat*}{2}
(i) &\;\;\;\;\;\;\;\;  \max_{y \sim x} \frac{w_{xy}}{\mu(x)} f(y) \leq &\asum{y \sim x} f(y) &< D_\mu f(x).  \\
(ii) && \asum{y \sim x} f^2(y) &<  D_\mu D_w f^2(x).\\
(iii)&&  \asum{y \sim x} f^4(y) &< D_\mu D_w^3 f^4(x). \\
\end{alignat*}
\end{lemma}

\begin{proof}
$(i)$ is obvious as $f > 0$.  $(ii)$ follows as
\[
\asum{y \sim x} f^2(y) \leq \frac{\mu(x)}{\min_{y \sim x} w_{xy}} \left( \asum{y \sim x} f(y) \right)^2  < D_\mu D_w f^2(x).
\]
(iii) follows similarly to (ii).  
\end{proof}

\subsection{Estimates on finite graphs}

We begin  by proving the gradient estimate in the compact case without boundary.   That is, we prove gradient estimates valid for positive solutions to parabolic equations on finite graphs.
%  We start with the simplest case, for positive solutions to the %heat equation and when the curvature is non-negative.

\begin{theorem}\label{thm:compactgre} Let $G$ be a finite graph satisfying $CDE(n,0)$, and let $u$ be a positive solution to the heat equation on $G$. Then for all $t>0$
\[ \frac{\Gamma(\sqrt{u})}{u} - \frac{\partial_t(\sqrt{u})}{\sqrt{u}} \leq \frac{n}{2t}\ .\]
\end{theorem}

\begin{proof}
Let 
\begin{equation}
\label{eq:F}
 F =t\left(\frac{2\Gamma(\sqrt{u})}{u} - \frac{2\partial_t(\sqrt{u})}{\sqrt{u}}\right)\ .
 \end{equation}

 Fix an arbitrary $T > 0$.  Our goal is to show that $F(x,T) \leq n$ for every $x \in V$. Let~$(x^*,t^*)$ be a maximum point of $F$ in $V \times [0,T]$. We may assume $F(x^*,t^*) > 0$. Hence $t^* > 0$. Moreover, by identity~\eqref{eq:keyformula}
 which is true both in the continuous and the discrete setting,
and the fact that $\LL u=0$ we know that
 \begin{equation}
 \label{eq:F-equivalence}
 F = t \cdot \frac{-2\Delta \sqrt{u} }{\sqrt{u}}\ . 
 \end{equation}
 
We conclude from~\eqref{eq:F-equivalence} that 
\begin{equation}
\label{ineq:super-harmonic-condition-verified}
(\Delta \sqrt{u}) (x^*,t^*) < 0\ .
\end{equation}

 In what follows all computations are understood to take place at the point $(x^*,t^*)$. We apply Lemma~\ref{lemma:L(gF)} with the choice of $g = u$. This gives
\[ \LL(u)\cdot F \geq \LL(u\cdot F) = \LL(t^*\cdot(2\Gamma(\sqrt{u}) - \Delta u)) = t^*\cdot \LL(2\Gamma(\sqrt{u}) - \Delta u) - (2\Gamma(\sqrt{u}) - \Delta u).\]
We know that $\LL(u) = 0$. Also, since $\Delta$ and $\LL$ commute,  $\LL(\Delta u) = 0$. So we are left with 
\begin{equation}\label{eq:gamma2} \frac{uF}{t^*}= 2\Gamma(\sqrt{u}) - \Delta u \geq t^* \cdot \LL(2\Gamma(\sqrt{u}))= t^* \cdot \left(2\Delta \Gamma(\sqrt{u}) - 4\Gamma(\sqrt{u}, \partial_t \sqrt{u})\right) = 4t^* \cdot \Gammatt(\sqrt{u})\ .\end{equation} 
The last equality is true by~\eqref{eq:Gamma2-equivalence} and since 
\[\partial_t \sqrt{u} = \frac{\partial_t u}{2\sqrt{u}} = \frac{ \Delta (\sqrt{u}^2)}{2\sqrt{u}}.\]
By~\eqref{ineq:super-harmonic-condition-verified}
and the $CDE(n,0)$-inequality applied to $\sqrt{u}(\cdot, t^*)$
 we get
\[ \frac{uF}{t^*} \geq \frac{4t^*}{n} \left(\Delta(\sqrt{u})\right)^2 \stackrel{\eqref{eq:F-equivalence}}{=} \frac{t^*}{n} \left( -\frac{\sqrt{u}F}{t^*}\right)^2 = \frac{u}{nt^*}F^2.\]

Indeed, the preceding line displays the reason why the identity~\eqref{eq:keyformula} is crucial: It allows us to relate $\mathcal{L}(uF)$ to $uF^2$.  

Thus we get $F \leq n$ at $(x^*,t^*)$ as desired.
\end{proof}

We can extend the result to the case of graphs satisfying $CDE(n,-K)$ for some $K > 0$ as follows.

\begin{theorem}\label{thm:compactgrecurv}
Let $G$ be a finite graph satisfying $CDE(n,-K)$ for some $K > 0$ and let $u$ be a positive solution to the heat equation on $G$. Fix $0 < \alpha < 1$. Then for all $t>0$
\[ \frac{(1-\alpha)\Gamma(\sqrt{u})}{u} - \frac{\partial_t(\sqrt{u})}{\sqrt{u}} \leq \frac{n}{(1-\alpha)2t} +  \frac{Kn}{\alpha}.\]
\end{theorem}

\begin{proof} We proceed similarly to the previous case, so we do not repeat computations that are exactly the same.
Let 
\[ F = t \cdot \frac{2(1-\alpha)\Gamma(\sqrt{u}) - \Delta u}{u} \leq t\cdot \frac{-2\Delta \sqrt{u}}{\sqrt{u}}. \]

Fix an arbitrary $T > 0$, and we will prove the estimate at $(x,T)$ for all $x \in V$.  As before let $(x^*,t^*)$ be the place where $F$ assumes its maximum in the $V \times [0,T]$ domain. We may assume $F(x^*,t^*) > 0$ otherwise there is nothing to prove. Hence $t^* > 0$ and $\Delta \sqrt{u} (x^*,t^*) < 0$.

 In what follows all computations are understood at the point $(x^*,t^*)$. 

We again apply Lemma~\ref{lemma:L(gF)} with the choice of $F = u$. As before, this gives
\[  0 =  \LL(u)\cdot F \geq \LL(u\cdot F) = \LL(t^* \cdot(2(1-\alpha)\Gamma(\sqrt{u}) - \Delta u)) = 4(1-\alpha)t^*\cdot \Gammatt(\sqrt{u}) - \frac{uF}{t^*}.\]
Applying the $CDE(n,-K)$ inequality to $\sqrt{u}$, multiplying by $t^*/u$ and rearranging gives
\[ F  \geq \frac{1-\alpha}{n}(F+\alpha G)^2 - 2(1-\alpha)t^*K G, \] where $G = t^* \cdot 2\Gamma(\sqrt{u})/u$. After expanding $(F+\alpha G)^2$ we throw away the $F\cdot G$ term, and use $\alpha^2 G^2$ to bound the last term on the right hand side. Completing the quadratic and linear term in $G$ to a perfect square yields

\begin{equation}\label{eq:beta} \alpha^2 G^2 - 2t^*Kn G \geq - \left(\frac{t^*Kn}{\alpha}\right)^2  = -(t^*)^2 C(\alpha, n ,K).
\end{equation}
So we have $F^2 \leq nF/(1-\alpha)  + t^2 C$, which implies 
\[ F(x,T) \leq F(x^*,t^*) \leq \frac{n}{1-\alpha}  + t^* \sqrt{C} \leq \frac{n}{1-\alpha} + T  \frac{Kn}{\alpha},\] which proves the gradient estimate at $(x,T)$ for all $x \in V$.  Since $T$ is arbitrary, we have the theorem as claimed.  
\end{proof}

We can also extend the result from solutions to the more general operator $(\mathcal{L}-q) = (\Delta - \partial_t - q)u = 0$, where $q(x,t)$ is a potential satisfying  $\Delta q \leq \vartheta$ and $\Gamma(q) \leq \eta^{2}$ for some $\vartheta\geq 0$ and $\eta \geq 0$. 

\begin{theorem} \label{thm:potential}
Let $G$ be a finite graph and $q(x,t): V \times \mathbb{R}^+ \to \mathbb{R}$ be a potential satisfying $\Delta q \leq \vartheta$ and $\Gamma(q) \leq \eta^{2}$ for all $x \in V$ and $t \geq 0$.  Suppose $u = u(x,t)$ satisfies $(\mathcal{L} - q)u = 0$ on $G$.  
\begin{enumerate}
\item If $G$ satisfies $CDE(n,0)$, then  for all $t>0$
\[ \frac{\Gamma(\sqrt{u})}{u} - \frac{\partial_t(\sqrt{u})}{\sqrt{u}} - \frac{q}{2} < \frac{n}{2t} + \frac{1}{2}\sqrt{n(\vartheta + \eta \sqrt{2D_\mu\left(D_w+1\right)} )}.\] 
\item Fix $0 < \alpha < 1$.  If $G$ satisfies $CDE(n,-K)$, for some $K \geq 0$,
then  for all $t>0$
\[
(1-\alpha) \frac{\Gamma(\sqrt{u})}{u} - \frac{\partial_t(\sqrt{u})}{\sqrt{u}} - \frac{q}{2} < \frac{n}{2(1-\alpha)t} + \frac{1}{2}C(\alpha,  K, n, \vartheta, \eta),
\]
where
\begin{align*}
&C(\alpha,  n, K, \vartheta, \eta)= \\ 
&\sqrt{ \frac{K^2n^2}{\alpha^2} + \frac{n}{1-\alpha}\left(\vartheta + \eta \left[ (1-\alpha)\sqrt{2D_\mu\left(D_w+1\right)} + \alpha \sqrt{2D_\mu\left(D_w^3+1\right)}\right] \right)}
\end{align*}
\end{enumerate}
\end{theorem}

\begin{proof}
Again, the proof is quite similar to the proof of Theorem \ref{thm:compactgre} so we do not repeat computations that are exactly the same.  
Let 
\[
F = t \cdot \left( \frac{2 \Gamma(\sqrt{u}) - u_t}{u} - q \right)
\]
As $(\Delta - \partial_t - q)u = 0$, note $u_t = \Delta u - qu$, so we may rewrite $F$ as
\[
F = t \cdot  \frac{2 \Gamma(\sqrt{u}) - \Delta u}{u} = -t \cdot \frac{2 \Delta \sqrt{u}}{\sqrt{u}} 
\]
as before.  

Again, we fix an arbitrary $T$ and take $(x^*, t^*)$ to be the place where $F$ assumes its maximum in the $V \times [0,T]$ domain, and we
may assume that $F(x^*,t^*) > 0$ and hence $t^*>0$ and $\Delta \sqrt{u}(x^*,t^*) < 0$.  All computations below should be understood
at the point $(x^*,t^*)$.  

We again apply Lemma \ref{lemma:L(gF)} with the choice that $F=u$.  The primary difference before is that in the application of Lemma \ref{lemma:L(gF)} is that at the maximum
\[
\mathcal{L}(uF) \leq \mathcal{L}(u)F = qu = -2t^*q\sqrt{u}\Delta \sqrt{u}.   
\]
Then, similarly as before,
\begin{align}
-2t^*q\sqrt{u}\Delta\sqrt{u} \geq \LL(uF) &= -\frac{uF}{t^*} + t^*\left( 2 \left[ \Delta( \Gamma(\sqrt{u})) - 2 \Gamma(\sqrt{u},\frac{u_t}{2\sqrt{u}}) \right] - \Delta \LL(u) \right)\nonumber \\
& = -\frac{uF}{t^*} + t^*\left( 4 \Gammatt(\sqrt{u}) + 2 \Gamma(\sqrt{u},q\sqrt{u}) - \Delta (qu)  \right) \label{eq:pot1}
\end{align}
Rearranging (\ref{eq:pot1}),
\begin{align}
0 \geq -\frac{uF}{t^*} + t^*\left( 4\Gammatt(\sqrt{u}) + 2 \Gamma(\sqrt{u},q\sqrt{u}) +2q\sqrt{u}\Delta\sqrt{u} - \Delta(qu) \right) \label{eq:pot2}
\end{align}
Note
\begin{align}
\Delta(qu) &= q\sqrt{u} \Delta \sqrt{u} + \sqrt{u}\Delta(q\sqrt{u}) + 2\Gamma(\sqrt{u},q\sqrt{u}) \nonumber \\&= 2q\sqrt{u}\Delta\sqrt{u} + u\Delta q + 2\sqrt{u}\Gamma(\sqrt{u},q) + 2\Gamma(\sqrt{u},q\sqrt{u}) 
\label{eq:pot3}
\end{align}
Combining $(\ref{eq:pot2})$ and $(\ref{eq:pot3})$, we obtain
\begin{align}
0 \geq -\frac{uF}{t^*} + t^*\left( 4 \Gammatt(\sqrt{u}) - u \Delta q - 2\sqrt{u}\Gamma(\sqrt{u},q)  \right) \label{eq:pot5}
\end{align}
Finally, we bound 
\[
2\Gamma(\sqrt{u},q) \leq \sqrt{2\Gamma(\sqrt{u})2\Gamma(q) }< \eta \sqrt{2 D_\mu \left( D_w+ 1\right) u}. 
\]
Here the first inequality follows from an application of Cauchy-Schwarz.  The bound on $\Gamma(\sqrt{u})(x^*,t^*)$ follows as $\Delta \sqrt{u}(x^*,t^*) < 0$, and applying Lemma 
\ref{prop:easy} (ii) yields
\begin{align*}
2\Gamma(\sqrt{u})(x^*,t^*) &= \asum{y \sim x^*} (\sqrt{u}(y,t^*)-\sqrt{u}(x^*,t^*))^2 \leq \asum{y \sim x^*} [u(y,t^*))+u(x^*,t^*)] \\&<  D_\mu \left(D_w+ 1\right)u(x^*,t^*).  
\end{align*}
With this,  $(\ref{eq:pot5})$ gives
\[
0 > -\frac{uF}{t^*} + t^*\left(4 \Gammatt(\sqrt{u}) - u \vartheta - \eta u \sqrt{2D_{\mu} \left(D_w + 1\right)} \right).
\]
Applying the $CDE(n,0)$ inequality, multiplying by $nt^*/u$ and rearranging yields 
\[
F^2 < nF - (t^*)^2n\left(\vartheta + \eta \sqrt{2D_\mu\left(D_w+1\right)}\right)
\]
which yields the first claim of the theorem, as above.  

The general case with negative curvature works by combining the above with the method of Theorem \ref{thm:compactgrecurv}.  

In the general case, 
\[
F = t\left(\frac{2(1-\alpha) \Gamma(\sqrt{u}) - u_t}{u} - q\right) = t \left( \frac{-2(1-\alpha)\sqrt{u}\Delta\sqrt{u} - \alpha \Delta u}{u} \right).
\]
Following the previous computation, again at $(x^*,t^*)$ maximizing $F$,
\begin{align*}
-2(1-\alpha)t^*q\sqrt{u}\Delta\sqrt{u} &- \alpha q\Delta u \geq \LL(uF)\nonumber\\&= -\frac{uF}{t^*} + t^*\left( 2(1-\alpha) \left[ \Delta( \Gamma(\sqrt{u})) - 2 \Gamma(\sqrt{u},\frac{u_t}{2\sqrt{u}}) \right] - \Delta \LL(u) \right)\nonumber \\
& = -\frac{uF}{t^*} + t^*\left( 4(1-\alpha) \Gammatt(\sqrt{u}) + 2(1-\alpha)\Gamma(\sqrt{u},q\sqrt{u}) - \Delta(qu)  \right)  
\end{align*}
After some computation and rearrangement, we get that 
\begin{align*}
0 &> -\frac{uF}{t^*} + t^*\left( 4(1-\alpha) \Gammatt(\sqrt{u}) - (1-\alpha)u\left(\vartheta + \eta\sqrt{2D_\mu \left(D_w+1\right)}\right) + \alpha(q\Delta u - \Delta(qu)) \right)\\
&= -\frac{uF}{t^*} + t^*\left( 4(1-\alpha) \Gammatt(\sqrt{u}) - (1-\alpha)u\left(\vartheta + \eta\sqrt{2D_\mu\left(D_w+1\right)}\right) - \alpha \left(u \Delta q + 2\Gamma(u, q)\right)\right). %\label{eqn:pot4}
 \end{align*}
By Lemma \ref{prop:easy} $(iii)$, and applying Cauchy-Schwarz we bound
\[
2\Gamma(u,q) \leq \sqrt{2\Gamma(u)2\Gamma(q)} < \eta u \sqrt{2D_\mu \left(D_w^3+1\right)},
\]
establishing that
\[
\left(u \Delta q + 2\Gamma(u, q)\right) < u \left( \vartheta +  \eta\sqrt{2D_\mu\left(D^3_w+1\right)} \right).
\]
Following the computations of the proof of Theorem \ref{thm:compactgrecurv} from $(\ref{eq:beta})$ we get
$F^{2} \leq n/(1-\alpha) F + t^{2}C^2(\alpha, n, K, \vartheta, \eta)$,
where 
\begin{align*}&C(\alpha,  n, K, \vartheta, \eta)= \\ 
&\sqrt{ \frac{K^2n^2}{\alpha^2} + \frac{n}{1-\alpha}\left(\vartheta + \eta \left[ (1-\alpha)\sqrt{2D_\mu\left(D_w+1\right)} + \alpha \sqrt{2D_\mu\left(D^3_w+1\right)}\right] \right)}
\end{align*}
Again, we prove the result for all $(x,T)$ but, as $T$ is arbitrary, this completes
the proof of the theorem.
\end{proof}
\subsection{General estimates in a ball}

We can prove somewhat weaker results in the presence of a boundary. We do not assume finiteness of the graph anymore, and we only assume the heat equation is satisfied in a finite ball. Our estimates will depend on the radius of this ball.

We shall prove two types of estimates. In this section we prove the first type that works for any non-negatively curved graph, while the second type requires the existence of so-called strong cut-off function on the graph that we will discuss later in Section \ref{strongcutoff}.

\begin{theorem}\label{thm:weakballgre} Let $G(V,E)$ be a (finite or infinite) graph and  $R > 0$, and fix $x_0 \in V$. 
\begin{enumerate}
\item Let $u: V \times \R \to \R$ a positive function such that $\LL u(x,t) = 0$ if  $d(x,x_0) \leq 2R$.  If $G$ satisfies $CDE(n,0)$ then for all $t>0$
\[ \frac{\Gamma(\sqrt{u})}{u} -\frac{ \partial_t \sqrt{u}}{\sqrt{u}}< \frac{n}{2t} + \frac{n(1+D_w)D_\mu}{R}\] in the ball of radius $R$ around $x_0$.
\item Let $u: V \times \R \to \R$ a positive function such that $(\LL-q) u(x,t) = 0$ if  $d(x,x_0) \leq 2R$, for some function $q(x,t)$ so that $\Delta q \leq \vartheta$ and $\Gamma(q) \leq \eta^2$.  
If $G$ satisfies $CDE(n,-K)$ for some $K > 0$, then for any $0 < \alpha < 1$ and all $t>0$
\[ \frac{(1-\alpha)\Gamma(\sqrt{u})}{u} -\frac{ \partial_t \sqrt{u}}{\sqrt{u}} -\frac{q}{2}< \frac{n}{(1-\alpha)2t} +  \frac{n(2+D_w)D_\mu}{(1-\alpha)R} + \frac{1}{2}C(\alpha,n,K,\vartheta,\eta),\]
where 
\begin{align*}&C(\alpha,  n, K, \vartheta, \eta)= \\ 
&\sqrt{ \frac{K^2n^2}{\alpha^2} + \frac{n}{1-\alpha}\left(\vartheta + \eta \left[ (1-\alpha)\sqrt{2D_\mu\left(D_w+1\right)} + \alpha \sqrt{2D_\mu\left(D^3_w+1\right)}\right] \right)}
\end{align*}
in the ball of radius $R$ around $x_0$.
\end{enumerate}
\end{theorem}

\begin{proof} First we consider the non-negative curvature case.
Let us define a cut-off function $\phi : V \to \R$ as
\[ \phi(v) = \left\{
\begin{array}{ccl}
0 &:& d(v,x_0) > 2R\\
\frac{2R-d(v,x_0)}{R}&:& 2R \geq d(v,x_0) \geq R\\
1 &:&R > d(v,x_0) 
\end{array}
\right.
\]
We are going to use the maximum-principle as in the proof of Theorem~\ref{thm:compactgre}. Let 
\[ F = t \phi\cdot \frac{2\Gamma(\sqrt{u})-\Delta u}{u}  = t\phi \cdot \frac{-2\Delta \sqrt{u}}{\sqrt{u}},\] and let $(x^*,t^*)$ be the place where $F$ attains its maximum in $V\times[0,T]$ for some arbitrary but fixed $T$. Our goal is to prove a bound on $F(x,T)$ for all $x \in V$ and as $T$ is arbitrary this completes the proof. This bound is positive, so we may assume that $F(x^*,t^*) > 0$. In particular this implies that $t^* >0$, $\phi(x^*) >0$, and $\Delta \sqrt{u}(x^*,t^*) < 0$. 

% and hence 
%\begin{equation}\label{eq:bounduy} u(y) < \deg(x^*)^2 u(x^*). \end{equation}

Let us first assume that $\phi(x^*) = 1/R$. Since positivity of $u$ implies that for any vertex $x$ 
\[ \frac{-\Delta \sqrt{u}}{ \sqrt{u}}(x) = \asum{y\sim x} \left( 1- \frac{\sqrt{u}(y)}{\sqrt{u}(x)} \right) \leq \frac{\deg(x^*)}{\mu(x^*)} \leq D_\mu,\]  we see that in this case $F(x^*,t^*) \leq 2t^*D_\mu/R$ and thus 
\[
F(x,T) \leq F(x^*,t^*) \leq 2t^*D_\mu/R \leq \frac{2TD_\mu}{R}.
\]
For $x \in B(x_0,R)$, $\phi \equiv 1$, so 
\[
F(x,T) = T \cdot \frac{\Gamma(\sqrt{u}) - \Delta u}{u}(x,T) \leq \frac{2TD_\mu}{R},
\]
and dividing by $T$ yields a stronger result than desired. We may therefore assume that $\phi(x^*) \geq \frac{2}{R}$ and $\phi$ does not vanish in the neighborhood of $x^*$.

Now we apply Lemma~\ref{lemma:L(gF)} with the choice of $F  = u/\phi$. Thus we get
\[\LL\left(\frac{u}{\phi}\right)F \geq \LL\left(\frac{u}{\phi}F\right) = -\frac{uF}{t^* \phi} + t^*\cdot \LL(2\Gamma(\sqrt{u})-\Delta u).
\]
Using the fact the $\LL(u) = 0$ we can write 
\[\LL\left(\frac{u}{\phi}\right) = \asum{y\sim x^*}\left( \frac{1}{\phi(y)} - \frac{1}{\phi(x^*)}\right)u(y).\] 
Using the same computation as in (\ref{eq:gamma2}) we get 
\[ t^*\cdot \LL(2\Gamma(\sqrt{u}) -\Delta u) = 4t^*\Gammatt(\sqrt{u}) \geq \frac{t^*}{n} (-2\Delta \sqrt{u})^2 = \frac{t^*}{n}\left( \frac{\sqrt{u}F}{t^*\phi}\right)^2.\]
Putting these together and multiplying through by $t^* \phi^2 /u$ we get 
\[ \phi(x^*)^2 t^* F\cdot\asum{y \sim x^*}\left( \frac{1}{\phi(y)} - \frac{1}{\phi(x^*)}\right)\frac{u(y)}{u(x^*)} + \phi F \geq \frac{1}{n}F^2.
\]
Let us write $\phi(x^*) = s/R$. Then for any $y \sim x^*$ we have $\phi(y) = (s \pm 1)/R$ or $\phi(y) = s/R$. In any case 
\[ \left| \frac{1}{\phi(y)} - \frac{1}{\phi(x^*)}\right| \leq \frac{R}{s(s-1)}. \] Using Lemma \ref{prop:easy} (ii) we have 
\begin{align*} \phi(x^*)^2 t^* F\cdot\asum{y \sim x^*}\left( \frac{1}{\phi(y)} - \frac{1}{\phi(x^*)}\right)\frac{u(y)}{u(x^*)} & \leq  \phi(x^*)^2 t^* F\cdot\asum{y \sim x^*}\left| \frac{1}{\phi(y)} - \frac{1}{\phi(x^*)}\right|\frac{u(y)}{u(x^*)} \\& \leq  \frac{2t^* F}{R}\cdot\asum{y \sim x^*} \frac{u(y)}{u(x^*)}   \\&< \frac{2t^*D_\mu D_w}{R} F.\end{align*}
Combining everything we can see that for any $x$ such that $d(x,x_0) \leq R$ and thus $\phi(x) =1$, at time $T$ 
\[ T \cdot \frac{2\Gamma(\sqrt{u})-\Delta u}{u}  = F(x,T) \leq F(x^*,t^*)< n\cdot \phi + \frac{2nt^*\deg^2(x^*)}{R\mu(x) w_{\mathrm{min}}} \leq n + \frac{2nTD_w D_\mu}{R},\] and dividing by $T$ gives the result.

The proof of the general case is simply the combination of the preceding proof with that of Theorem~\ref{thm:potential}.
\end{proof}

\begin{corollary} If $G(V,E)$ is an infinite, bounded degree graph satisfying $CDE(n,0)$ and $u$ is a positive solution to the heat equation on $G$, then 
\[ 
 \frac{\Gamma(\sqrt{u})}{u} -\frac{ \partial_t \sqrt{u}}{\sqrt{u}} \leq \frac{n}{2t}
\] on the whole graph.
\end{corollary}

\subsection{Strong cut-off functions}\label{strongcutoff}

In the case of manifolds~\cite{LiYau}, a result similar to Theorem~\ref{thm:weakballgre} holds with $1/R^2$ instead of $1/R$.  In one of the key steps of the argument the Laplacian comparison theorem is applied to the distance function. This together with the chain rule implies that one can find a cut-off function $\phi$ that satisfies 
\[\Delta \phi \geq -c(n) \frac{1 +R\sqrt{K}}{R^2},\] where $c$ is a constant that only depends on the dimension. Since the cut-off function $\phi$ also satisfies 
\begin{equation}\label{eq:strong1}\frac{|\nabla\phi|^2}{\phi}<\frac{c(n)}{R^2}\end{equation} it follows, that there exists a constant $C(n)$, that only depends on the dimension such that 
\begin{equation}\label{eq:strong2}\Delta \phi - 2\frac{|\nabla \phi|^2}{\phi} \geq -C(n)\frac{1+R\sqrt{K}}{R^2}\end{equation}

Unfortunately on graphs the Laplacian comparison theorem for the usual  graph distance is not true - think for instance of the lattice $\mathbb{Z}^2$. This is the reason why in general we have to assume the existence of a cut-off function that has similar properties to \eqref{eq:strong1} and \eqref{eq:strong2}, in order to prove a gradient estimate with $1/R^2$. Noting that for a diffusion semigroup  and hence in particular for the Laplace-Beltrami operator on manifolds
\[\phi^2\Delta\frac{1}{\phi} = -\Delta \phi + 2 \frac{\Gamma(\phi)}{\phi}\leq C(n)\frac{1+R\sqrt{K}}{R^2}\] and
\[\phi^3\Gamma\left(\frac{1}{\phi}\right)= \frac{\Gamma(\phi)}{\phi}\leq \frac{C(n)}{R^2}\]
this discussion motivates the following definition:

\begin{definition}\label{def:strong_cutoff}
 Let $G(V,E)$ be a graph satisfying $CDE(n,-K)$ for some $K \geq 0$. We say that the function $\phi: V\to [0,1]$ is an \textit{$(c,R)$-strong} cut-off function centered at $x_0 \in V$ and supported on a set $S \subset V$ if $\phi(x_0 ) =1$, $\phi(x) = 0 $ if $x \not \in S$ and for any vertex $x \in S$ 
\begin{enumerate}
\item either $\phi(x) <\frac{c(1 + R\sqrt{K})}{2R^2}$,
\item or $\phi$ does not vanish in the immediate neighborhood of $v$ and 
\[\phi^2(x) \Delta \frac{1}{\phi}(x) < D_\mu \frac{c(1 + R\sqrt{K})}{R^2} \;\mbox{ and }\; \phi^3(x) \Gamma\left(\frac{1}{\phi}\right)(x) < D_\mu \frac{c}{R^2},\]
where the constant $c = c(n)$ only depends on the dimension $n$.  
\end{enumerate}
\end{definition}
\begin{remark}
The `strength' of the strong cutoff function depends on the size of support $S$.  In
order to get results akin to those in the manifold case, with $\frac{1}{R^2}$ appearing
for solutions valid in $B(x_0, cR)$ one requires a strong cutoff function whose
support lies within a ball of radius $cR$.  The cutoff function defined above, using graph distance, gives a strong cutoff function on the ball of radius $R^2$.  Theorem 
\ref{thm:strongballgre} yields a better estimate than Theorem \ref{thm:weakballgre} whenever one can find a strong cutoff function with support in a ball of radius $\ll R^2$.  

In Section~\ref{sec:examples} we will show (see Corollary~\ref{corol:zd} and Proposition~\ref{prop:zdstrong}) that the usual Cayley graph of $\Z^d$ with the regular or the normalized Laplacian satisfies $CDE(2d,0)$ and admits a $(100,R)$-strong cut-off function supported on a ball of radius $\sqrt{d}R$ centered at $x_0$.
\end{remark}

\begin{theorem} \label{thm:strongballgre} Let $G(V,E)$ be a (finite or infinite) graph satisfying $CDE(n,-K)$ for some $K \geq 0$.  Let $R > 0$ and fix $x_0 \in V$. Assume that $G$ has a $(c,R)$-strong cut-off function supported on $S \subset V$ and centered at $x_0$. Fix $0<\alpha <  1$.  Let $u: V \times \R \to \R$ a positive function such that $(\LL-q)u(x,t) = 0$ if  $x \in S$, for some $q(x,t)$ satisfying $\Delta q \leq \vartheta$ and $\Gamma(q) \leq \eta^2.$
Then for every $\epsilon \in (0,1)$
\begin{align*} &\left(\frac{(1-\alpha)\Gamma(\sqrt{u})}{u} -\frac{ \partial_t \sqrt{u}}{\sqrt{u}} -\frac{q}{2}\right)(x_0,t)  \\ <&  \frac{n}{2(1-\alpha)t} + \frac{D_\mu cn}{2(1-\alpha)R^2}\left(1 +R\sqrt{K}+\frac{n (D_w+1)^2}{4\alpha (1-\alpha)}\right) +\frac{1}{2} C(\alpha, n, K, \vartheta, \eta,\epsilon),
\end{align*} 
where 
\[C(\alpha,  n, K, \vartheta, \eta,\epsilon)= \sqrt{\frac{n}{1-\alpha}\vartheta +   \frac{K^2n^2}{(1-\epsilon)\alpha^2} + \left(\frac{n(1 +\alpha D_w)\eta}{(1-\alpha)\alpha^{1/2}\epsilon^{1/4}} \right)^{\frac{4}{3}}}. \] \end{theorem}
As we noted in the remark above, the lattice $\mathbb{Z}^d$ yields a $(c,R)$-strong cutoff function in the ball $B(x_0,\sqrt{d}R)$ and $CDE(0,2d)$.  As a result Theorem \ref{thm:strongballgre} specializes to the following.  
\begin{corollary}  If $u$ is a solution of the heat equation $\LL u = 0$ in  $B(x_0,\sqrt{d}R)$, then (with the choice of $\alpha = 1/2$):
\[ \frac{\Gamma(\sqrt{u}) - \Delta u}{u}(x_0,t) \leq \frac{4d}{t} + \frac{c(d)}{R^2}\] for some explicit constant $c(d)$ depending on the dimension.
\end{corollary}

\begin{proof}[Proof of Theorem~\ref{thm:strongballgre}]
We proceed similarly to the proof of Theorem~\ref{thm:weakballgre}, except that we assume $\phi$ is a $(c,R)$-strong cut-off function centered at $x_0$. Let us choose
\[ F = t\phi \cdot \frac{2(1-\alpha)\Gamma(\sqrt{u}) - \Delta u}{u},\] and let $(x^*,t^*)$ denote the place where $F$ attains its maximum in $V\times [0,T]$ for some arbitrary but fixed $T$. Again, our goal is to show that $F(x,T)$ is bounded for 
all $x \in V$, and since $T$ is arbitrary this completes the result.  We bound $F$ by some positive quantity, hence we may assume $F(x^*,t^*) > 0$. This implies $t^* >0, \phi(x^*) >0$, and $2\Gamma(\sqrt{u})-\Delta u \geq 2(1-\alpha)\Gamma(\sqrt{u}) -\Delta u > 0$ at $(x^*,t^*)$. Hence $\Delta \sqrt{u}(x^*,t^*) > 0$ as in the proof of Theorem~\ref{thm:weakballgre}.

First, if $\phi(x^*) \leq \frac{c(1 + R\sqrt{K})}{2R^2}$ then we are done, since 
\[ \frac{2(1-\alpha)\Gamma(\sqrt{u}) -\Delta u}{u} \leq \frac{2\Gamma(\sqrt{u}) -\Delta u}{u} = \frac{-2\Delta \sqrt{u}}{\sqrt{u}} \leq 2D_\mu,\]
as we have seen in the proof of Theorem~\ref{thm:weakballgre}. Thus we may assume that Case 2 of Definition~\ref{def:strong_cutoff} holds.

In what follows all equations are to be understood at $(x^*,t^*)$. We use Lemma~\ref{lemma:L(gF)} with the choice of $F = u/\phi$ to get
\begin{align}
\LL\left(\frac{u}{\phi}\right) F \geq \LL\left(\frac{u}{\phi}F\right) &= -\frac{uF}{t^* \phi} + t^* \cdot \LL(2(1-\alpha)\Gamma(\sqrt{u}) -\Delta u) \nonumber
\\&= -\frac{uF}{t^* \phi} + t^* \cdot \left[ (1-\alpha)\LL(2\Gamma(\sqrt{u})) - \Delta (qu)\right] \nonumber
\\ &= -\frac{uF}{t^* \phi} + t^* \cdot \left[ 4(1-\alpha)\Gammatt(\sqrt{u}) + 2(1-\alpha)\Gamma(\sqrt{u},\sqrt{u}q) - \Delta(qu)\right].
\label{eq:alpha}
\end{align}

On the left hand side we use Cauchy-Schwarz:
\begin{align} \LL\left( \frac{u}{\phi}\right) &= \frac{\LL(u)}{\phi} + \LL\left(\frac{1}{\phi}\right) u + 2\Gamma\left( \frac{1}{\phi}, u \right)\nonumber \\ &= \frac{qu}{\phi} + u \Delta \frac{1}{\phi} + 2\Gamma\left( \frac{1}{\phi}, u \right)\nonumber \\ &\leq \frac{qu}{\phi} + u\Delta \frac{1}{\phi} + 2\sqrt{\Gamma\left(\frac{1}{\phi}\right)} \sqrt{\Gamma(u)}, \label{eq:boundlhs}
\end{align} since $\LL(u) = qu$. 

Collecting the $q$-terms in (\ref{eq:alpha}) and using (\ref{eq:boundlhs}), we observe
that they are 
\begin{align}
\nonumber &t^*\left[2(1-\alpha)\Gamma(\sqrt{u},\sqrt{u}q)-\Delta(qu)\right] - \frac{qu}{\phi}F \\&\nonumber =
t^*\left[(1-\alpha)\left(2\Gamma(\sqrt{u},\sqrt{u}q)-\Delta(qu)-2q\sqrt{u}\Delta\sqrt{u}\right) + \alpha\left(q\Delta(u) - \Delta(qu)\right)\right]\\\nonumber 
 &> -ut^*\left(\vartheta + 2(1-\alpha)\eta\frac{\sqrt{\Gamma(\sqrt{u})}}{\sqrt{u}} + 2\alpha \eta \frac{\sqrt{\Gamma(u)}}{u}\right)\\
& \geq -ut^*\left(\vartheta +2\eta (1+\alpha D_w)\frac{\sqrt{\Gamma(\sqrt{u})}}{\sqrt{u}}\right)\label{eq:qterms}
\end{align}
In the computation above we used several times Cauchy-Schwarz, \eqref{eq:pot3} and the observation that
$\Gamma(u)/u^2$ can be controlled by $\Gamma(\sqrt{u})/u$ in the following way: By Lemma
\ref{prop:easy} (i), and the fact that $\Delta \sqrt{u}(x^*) < 0$, we have that 
$\sqrt{u}(y) < D_w \sqrt{u}(x^*)$ for any $y \sim x^*$. Hence
\begin{align}
\frac{2\Gamma(u)}{u^2} = \asum{y \sim x^*} \left(1 - \frac{u(y)}{u(x^*)}\right)^2 &= \asum{y\sim x^*} \left(1 - \frac{\sqrt{u}(y)}{\sqrt{u}(x^*)}\right)^2 \left(1 + \frac{\sqrt{u}(y)}{\sqrt{u}(x^*)}\right)^2 \nonumber \\ &< \left(D_w +1\right)^2 \frac{2\Gamma(\sqrt{u})}{u}. \label{eq:degx}
\end{align}

Combining \eqref{eq:qterms} with (\ref{eq:alpha}) and multiplying by $t^*\phi^2 / u$ we get
\begin{multline}\label{eq:sqrt}
 (t^*)^2\phi^2\left(\vartheta +2\eta (1+\alpha D_w)\frac{\sqrt{\Gamma(\sqrt{u})}}{\sqrt{u}}\right)+ Ft^* \phi^2\Delta \frac{1}{\phi} + Ft^*\sqrt{2\phi^3\Gamma\left(\frac{1}{\phi}\right)} \sqrt{\phi\frac{2\Gamma(u)}{u^2}} + \phi F \\ >  4(1-\alpha)\frac{\Gammatt(\sqrt{u})}{u}(t^*)^2\phi^2 .
 \end{multline}

Let us introduce the notation $G = 2t^*\phi \Gamma(\sqrt{u})/u$. Using (\ref{eq:degx}), and that $\phi$ is a $(c,R)$-strong cut-off function we can further estimate the left hand side of (\ref{eq:sqrt}) from above:
\begin{align}\label{eq:G}\nonumber
 (t^*)^2\phi^2\vartheta +& \sqrt{2}\eta(1+\alpha D_w) (t^*\phi)^{\frac{3}{2}}\sqrt{G}+
\frac{t^*D_\mu c(1 + R\sqrt{K})}{R^2}F + \phi F \\+& \sqrt{2}\left(D_w+1\right)\left(\frac{t^*D_\mu c}{R^2}\right)^{\frac{1}{2}}F\sqrt{G} > 4(1-\alpha)\frac{\Gammatt(\sqrt{u})}{u}(t^*)^2 \phi^2.
\end{align}

Using that the graph satisfies $CDE(n,-K)$ we can write
\[ 4(t^*)^2\phi^2 \frac{\Gammatt(\sqrt{u})}{u} \geq \frac{1}{n}\left(t^*\phi \frac{2\Delta \sqrt{u}}{\sqrt{u}}\right)^2 - 2K (t^*)^2\phi^2 \frac{2\Gamma(\sqrt{u})}{u} = \frac{(F +\alpha G)^2}{n} - (2t^*\phi)K G.
\] 
Combining with (\ref{eq:G}) we have
\begin{align*}
&\frac{n}{1-\alpha}\left(  (t^*)^2\phi^2\vartheta + \sqrt{2}\eta(1+\alpha D_w) (t^*\phi)^{\frac{3}{2}}\sqrt{G}\right)\\&+ \frac{n}{1-\alpha}\left(\frac{t^*D_\mu c(1 + R\sqrt{K})}{R^2} + \phi + \sqrt{2}\left(D_w+1\right)\left(\frac{t^*D_\mu c}{R^2}\right)^{\frac{1}{2}} \sqrt{G} \right) F \\ > &F^2 + 2\alpha FG  + \alpha^2 G^2 -  2t^*\phi KnG
\end{align*}

Notice that completing the left hand side to a  to a perfect square gives 
\[ 2\alpha GF - \sqrt{2}\left(D_w+1\right) \frac{n}{(1-\alpha)}\left(\frac{t^*D_\mu c}{R^2}\right)^{\frac{1}{2}} \sqrt{G}F \geq - \frac{t^*D_\mu c}{R^2}\left(D_w+1\right)^2 \frac{n^2}{4\alpha(1-\alpha)^2}F
\]
and hence
\begin{align}\nonumber
&\frac{n}{1-\alpha}\left( (t^*)^2\phi^2\vartheta +\sqrt{2}\eta(1+\alpha D_w) (t^*\phi)^{\frac{3}{2}}\sqrt{G}\right) \\&+ \frac{n}{1-\alpha}\left(\frac{t^*D_\mu c(1 + R\sqrt{K})}{R^2} + \phi + (D_w+1)^2 \frac{t^*D_\mu c n}{4\alpha (1-\alpha)R^2}\right)F \\ &> F^2 + \alpha^2 G^2 - 2t^*\phi KnG.\label{eq:G2}
\end{align}
Now we cosider the terms in $G$
\begin{align*}\alpha^2 G^2 - 2t^*\phi KnG - \frac{n}{1-\alpha}\sqrt{2}\eta(1+\alpha aD_w) (t^*\phi)^{\frac{3}{2}}\sqrt{G}
\end{align*}
Defining $F = 2 \frac{\Gamma(\sqrt{u})}{\sqrt{u}}$ we obtain for every $\epsilon \in (0,1)$
\begin{align*}(t^*\phi)^2&\left( \alpha^2F -(1-\epsilon) \alpha^2F + (1-\epsilon) \alpha^2F -2KnF -\frac{n}{1-\alpha}\sqrt{2}\eta(1+\alpha D_w) \sqrt{F}  \right)\\\geq & (t^*\phi)^2\left(\epsilon \alpha^2 F^2 - \frac{K^2n^2}{(1-\epsilon)\alpha^2} -\frac{n}{1-\alpha}\sqrt{2}\eta(1+\alpha D_w) \sqrt{F}\right)\\\geq& (t^*\phi)^2\left( - \frac{K^2n^2}{(1-\epsilon)\alpha^2} - \left(\frac{n(1 +\alpha D_w)\eta}{(1-\alpha)\alpha^{1/2}\epsilon^{1/4}} \right)^{\frac{4}{3}}\right)
\end{align*}

This combined with (\ref{eq:G2}) now yields
\begin{align*}
\frac{n}{1-\alpha}& \left( \frac{t^*D_\mu c(1 + R\sqrt{K})}{R^2} + \phi + (D_w+1)^2\frac{t^*D_\mu c n}{4\alpha (1-\alpha)R^2}\right)F\\ &+  (t^*\phi)^2 
\left[\frac{n}{1-\alpha}\vartheta +   \frac{K^2n^2}{(1-\epsilon)\alpha^2} + \left(\frac{n(1 +\alpha D_w)\eta}{(1-\alpha)\alpha^{1/2}\epsilon^{1/4}} \right)^{\frac{4}{3}} \right] 
\geq F^2, 
\end{align*}
which easily implies
\[
F < \frac{n}{1-\alpha}\left( \frac{t^*D_\mu c(1 + R\sqrt{K})}{R^2} + \phi + \left(D_w+1\right)^2\frac{t^*D_\mu c n}{4\alpha (1-\alpha)R^2} \right) +  t^*\phi C(\alpha,  n, K, \vartheta, \eta, \epsilon)
\]
where
\begin{align*}C(\alpha,  n, K, \vartheta, \eta,\epsilon)= \sqrt{\frac{n}{1-\alpha}\vartheta +   \frac{K^2n^2}{(1-\epsilon)\alpha^2} + \left(\frac{n(1 +\alpha D_w)\eta}{(1-\alpha)\alpha^{1/2}\epsilon^{1/4}} \right)^{\frac{4}{3}}}.
\end{align*}

Using that $\phi \leq 1,$  $\phi(x_0) =1$,  $t^* \leq T$, and $F(x_0,T) \leq F(x^*,t^*)$, and finally dividing by $T$
 we get the desired upper bound 
 \begin{align*} &\frac{(1-\alpha)2\Gamma(\sqrt{u}) - \Delta u}{u}(x_0,T)\\ &< \frac{n}{1-\alpha}\left( \frac{D_\mu c(1 + R\sqrt{K})}{R^2} + \frac{1}{T} + \left(D_w+1\right)^2\frac{D_\mu c n}{4\alpha (1-\alpha)R^2}\right) +    C(\alpha, n, K, \vartheta, \eta,\epsilon).
\end{align*}

\end{proof}

%==================================================
\section{Harnack inequalities} \label{harnack}
%================================================
In this section we explain how the gradient estimates can be used to derive Harnack-type inequalities. The proof is based on the method used by Li and Yau in~\cite{LiYau}, though the discrete space does pose some extra difficulty.

In order to state the result in complete generality (in particular, when $f$ is a solution to $(\mathcal{L}-q)f = 0$ as opposed to a solution to the heat equation), we need to introduce
a discrete analogue of the Agmon distance between two points $x$, and $y$ which are connected in $B(x_0,R)$.  For a path $p_0p_1\dots p_k$ define the length of the path to be $\ell(P) = k$.  Then in a graph with maximum measure $\mu_{\mathrm{max}}$:
\begin{multline*}
\varrho_{q,x_0,R,\mu_{\mathrm{max}}, w_{\mathrm{min}}, \alpha}(x,y,T_1,T_2) = \inf \left\{ \frac{2\mu_{\mathrm{max}}\ell^2(P)}{w_{\mathrm{min}}(1-\alpha)(T_2-T_1)}\right. \\+ \left. \sum_{i=0}^{k-1} \left( \int_{t_i}^{t_{i+1}} q(x_i,t)dt + \frac{k}{(T_2-T_1)^2}\int_{t_i}^{t_{i+1}} (t-t_i)^2(q(x_i,t) - q(x_{i+1},t)) dt \right) \right\}
%\right\}.
\end{multline*}
where the infinum is taken over the set of all paths $P = p_0p_1p_2p_3\dots p_k$ so that
$p_0 = x$, $p_k = y$ and having all $p_i \in B(x_0,R)$, and the times $T_1=t_0 , t_1, t_2, \dots, t_{k}=T_2$ evenly divide the interval $[T_1,T_2]$. In the case when the graph satisfies $CDE(n,0)$ one can set $\alpha =0$.

\begin{remark}
In the special case where $q \equiv 0$ and $R = \infty$, which will arise when $f$ is a 
solution to the heat equation on the entire graph, then $\varrho$ simplifies drastically.  In particular, 
\[
\varrho_{\mu_{\mathrm{max}},\alpha, w_{\mathrm{min}}}(x,y,t_1,t_2) = \frac{2\mu_{\mathrm{max}} d(x,y)^2}{(1-\alpha)(T_2-T_1)w_{\mathrm{min}}},
\]
where $d(x,y)$ denotes the usual graph distance. 
\end{remark}

\begin{theorem} Let $G(V,E)$ be a graph with measure bound $\mu_{\mathrm{max}}$, and suppose that a function $f: V \times \R \to \R$ satisfies
\[ (1-\alpha)\frac{\Gamma(f)}{f^2}(x,t) - \frac{\partial_t f}{f}(x,t) - q(x,t) \leq \frac{c_1}{t} + c_2 \] 
whenever $x \in B(x_0,R)$ for $x_0 \in V$ along with some $R \geq 0$, some $0 \leq \alpha < 1$ and positive constants $c_1,c_2$. Then for $T_1 < T_2$ and $x,y \in V$ we have
\[ f(x,T_1) \leq f(y,T_2) \left(\frac{T_2}{T_1}\right)^{c_1}\cdot \exp \left( c_2(T_2 - T_1) + \varrho_{q,x_0,R,\mu_{\mathrm{max}}, w_{\mathrm{min}}, \alpha}(x,y,T_1,T_2) \right)
\] \label{thm:harnack}
\end{theorem}
In the case of unweighted graphs, and when dealing with positive solutions to the heat
equation everywhere, Theorem \ref{thm:harnack} simplifies greatly.  
\begin{corollary} Suppose $G(V,E)$ is a finite or infinite unweighted graph satisfying $CDE(n,0)$, and $\mu(x) = \deg(x)$ for all vertices $x \in V$.  If $u$ is a positive solution to the heat equation on $G$, then
\[
u(x,T_1) \leq u(y,T_2) \left( \frac{T_2}{T_1}\right)^{n} \exp\left(\frac{4D d(x,y)^2}{T_2-T_1}\right),
\]  
where $D$ denotes the maximum degree of a vertex in $G$.  
\end{corollary}
\begin{remark}
Observe that in the application of Theorem $\ref{thm:harnack}$ to prove the corollary,
one may take $c_1 = \frac{n}{2}$ (see Theorem \ref{thm:potential} and Theorem \ref{thm:weakballgre}), but Theorem $\ref{thm:harnack}$ naturally compares
$\sqrt{u(x,T_1)}$ to $\sqrt{u(x,T_2)}.$  To compare $u(x,T_1)$ to $u(x,T_2)$ requires 
squaring both sides and introduces a factor of two in the exponent.
\end{remark}
Before we give the proof of Theorem \ref{thm:harnack}, we need one simple lemma.

\begin{lemma} For any $c > 0$ and any functions $\psi, q_1, q_2 : [T_1,T_2] \to \R$, we have  \label{lem:ineq} 
\begin{align*} \min_{s \in [T_1,T_2]}  \psi(s) &- \frac{1}{c} \int_s^{T_2} \psi^2(t)dt + \int_{T_1}^{s} q_1(t) dt + \int_{s}^{T_2} q_2(t)dt \  \\&\leq \frac{c}{T_2 - T_1} + 
\int_{T_1}^{T_2} q_1(t)dt + \frac{1}{(T_2-T_1)^2}\int_{T_1}^{T_2} (t-T_1)^2(q_2(t) - q_1(t)) dt.
\end{align*}
\end{lemma}
\begin{proof} 
We bound  the minimum by an averaged sum. Let $\phi(t) = \frac{2}{c}(t-T_1)$. Then
\begin{align*}
\min_{s \in [T_1,T_2]} \psi(s) -& \frac{1}{c}\int_s^{T_2} \psi^2(t)dt + \int_{T_1}^{s} q_1(t) dt + \int_{s}^{T_2} q_2(t)dt  \\\leq & \frac{\int_{T_1}^{T_2} \phi(s)\left(\psi(s) - \frac{1}{c}\int_s^{T_2} \psi^2(t)dt +  \int_{T_1}^{s} q_1(t) dt + \int_{s}^{T_2} q_2(t)dt\right)ds}{\int_{T_1}^{T_2} \phi(s)ds} \\ =& \frac{c}{(T_2-T_1)^2 }\left(\int_{T_1}^{T_2} \phi(s)\psi(s)ds - \frac{1}{c}\int_{T_1}^{T_2} \psi^2(t) \int_{T_1}^t \phi(s)ds dt\right. \\
&\;\;\;\;\;\;\;\;\;\;\left.+ \int_{T_1}^{T_2} q_1(t) \int_{t}^{T_2} \phi(s)dsdt + \int_{T_1}^{T_2} q_2(t) 
\int_{T_1}^{t} \phi(s)dsdt \right) \\=&\frac{c}{(T_2-T_1)^2} \left[\int_{T_1}^{T_2} \left(2\frac{t-T_1}{c}\psi(t) - \psi^2(t) \left(\frac{t-T_1}{c}\right)^2\right) dt\right.  \\
&\;\;\;\;\;\;\;\;\;\; \left. + \int_{T_1}^{T_2} \frac{(T_2-T_1)^2 - (t-T_1)^2}{c} q_1(t) dt + \int_{T_1}^{T_2} \frac{(t-T_1)^2}{c} q_2(t)dt \right] 
\\& \leq \frac{c}{T_2-T_1} + \int_{T_1}^{T_2} q_1(t)dt + \frac{1}{(T_2-T_1)^2}\int_{T_1}^{T_2} (t-T_1)^2(q_2(t) - q_1(t)) dt.
\end{align*} as we claimed, since $2x - x^2 \leq 1$.
\end{proof}

With this, we can return to the proof of Theorem \ref{thm:harnack}.

\begin{proof}[Proof of Theorem \ref{thm:harnack}]
Let us first assume that $x \sim y$. Then for any $s \in [T_1, T_2]$ we can write
\begin{align*} 
\log f(x,T_1) - \log f(y,T_2) =& \log \frac{f(x,T_1)}{f(x,s)} + \log \frac{f(x,s)}{f(y,s)} + \log \frac{f(y,s)}{f(y,T_2)}  \\=& - \int_{T_1}^s \partial_t \log f(x,t) dt + \log \frac{f(x,s)}{f(y,s)} - \int_{s}^{T_2} \partial_t \log f(y,t)dt
\end{align*}
We use the assumption that 
\[ 
-\partial_t \log f = -\frac{\partial_t f}{f} \leq  \frac{c_1}{t} + c_2 - (1-\alpha)\frac{\Gamma(f)}{f^2}  + q 
\] 
to deduce
\begin{align*} 
&\log f(x,T_1) - \log f(y,T_2) \\ \leq &\int_{T_1}^{T_2} \frac{c_1}{t} + c_2 dt - (1-\alpha)\left( \int_{T_1}^s \frac{\Gamma(f)}{f^2}(x,t) dt + \int_{s}^{T_2} 
 \frac{\Gamma(f)}{f^2}(y,t) dt \right) + \log \frac{f(x,s)}{f(y,s)} \\& + \int_{T_1}^{s} q(x,t) dt + \int_{s}^{T_2} q(y,t) dt \\ \leq&
 c_1\log\frac{T_2}{T_1} + c_2(T_2 - T_1) - \frac{(1-\alpha)w_{\mathrm{min}}}{2\mu_{\mathrm{max}}} \int_s^{T_2} \left|\frac{f(y,t)-f(x,t)}{f(y,t)}\right|^2 + \frac{f(x,s) - f(y,s)}{f(y,s)}
\\& + \int_{T_1}^{s} q(x,t) dt + \int_{s}^{T_2} q(y,t) dt.
\end{align*}
In the second step we threw away the $\int_{T_1}^s$ term, and used  $\Gamma(f)(y,t) \geq \frac{1}{2}w_{\mathrm{min}}(f(y,t)-f(x,t))^2/\mu_{\mathrm{max}}$ as well as the fact that $\log r \leq r-1$ for any $r\in \R$.

We are free to choose the value of $s$ for which the right hand side is minimal.  We use Lemma \ref{lem:ineq}, with the choice of $\psi(t) = f(x,t)/f(y,t) - 1$ and $c = (1- \alpha)w_{\mathrm{min}}/2\mu_{\mathrm{max}}$ along with $q_1(t) = q(x,t)$ and $q_2(t) = q(y,t)$ to get
\begin{align}
\log f(x,T_1) - \log f(y,T_2) \leq & c_1\log\frac{T_2}{T_1} + c_2(T_2 - T_1) + \frac{2\mu_{\mathrm{max}}}{(1-\alpha)(T_2 - T_1)w_{\mathrm{min}}} \nonumber\\ & + \int_{T_1}^{T_2} q(x,t)dt + \frac{1}{(T_2-T_1)^2}\int_{T_1}^{T_2} (t-T_1)^2(q(y,t) - q(x,t)) dt. \label{eq:adjacent}
\end{align}

To handle the case when $x$ and $y$ are not adjacent, simply let $x = x_0, x_1, \dots, x_k = y$ denote a path $P$ between $x$ and $y$ entirely within $B(x_0,R)$, and let $T_1 = t_0 < t_1 < \dots < t_k = T_2$ denote a subdivision of the time interval $[T_1,T_2]$ into $k$ equal parts. 
For any $0 \leq i \leq k-1$ we can use (\ref{eq:adjacent}) to get
\begin{align*}
 \log f(x,T_1) -& \log f(y,T_2) \\ =& \sum_{i = 0}^{k-1} \big[\log f(x_i,t_i) - \log f(x_{i+1},t_{i+1})\big]  \\ \leq& \sum_{i = 0}^{k-1} \left( c_1\log \frac{t_{i+1}}{t_i} + c_2(t_{i+1}-t_i) + \frac{2\mu_{\mathrm{max}}}{(1-\alpha)\frac{T_2 - T_1}{k} w_{\mathrm{min}}} \right)  \\ & + \sum_{i=0}^{k-1} \left(\int_{t_i}^{t_{i+1}} q(x_i,t)dt + \frac{k}{(T_2-T_1)^2}\int_{t_i}^{t_{i+1}} (t-t_i)^2(q(x_i,t) - q(x_{i+1},t)) dt \right)
\\
 \leq &  c_1 \log \frac{T_2}{T_1} + c_2(T_2 - T_1) + \frac{2k^2 \mu_{\mathrm{max}}}{(1-\alpha)(T_2 - T_1)w_{\mathrm{min}}} \\
 & +  \sum_{i=0}^{k-1} \left( \int_{t_i}^{t_{i+1}} q(x_i,t)dt + \frac{k}{(T_2-T_1)^2}\int_{t_i}^{t_{i+1}} (t-t_i)^2(q(x_i,t) - q(x_{i+1},t)) dt. \right)
 \end{align*}
 Minimizing all paths, we have that
 \[
 \log f(x,T_1) - \log f(y,T_2)  \leq c_1 \log \frac{T_2}{T_1} + c_2(T_2-T_1) + \varrho_{q,x_0,R,\mu_{\mathrm{max}},w_{\mathrm{min}}, \alpha}(x,y,T_1,T_2).
 \]
Hence
\[ f(x,T_1) \leq f(y,T_2) \left(\frac{T_2}{T_1}\right)^{c_1}\cdot \exp \left( c_2(T_2 - T_1) + \varrho_{q,x_0,R,\mu_{\mathrm{max}},w_{\mathrm{min}}, \alpha}(x,y,t_1,t_2) \right)
\] as was claimed.

\end{proof}

\section{Examples}\label{sec:examples}

In this section we show that our curvature notion behaves somewhat as expected, by computing curvature lower bounds for certain classes of graphs. We also show that $\Z^d$ admits strong cut-off functions in the sense of Definition~\ref{def:strong_cutoff}.

\subsection{General graphs and trees}

Here we prove that every graph satisfies $CDE\left(2,-D_\mu\left(\frac{D_w}{2}+1\right)\right)$. We show that this bound is close to sharp for graphs that are locally trees, in particular the curvature of a $D$-regular large girth graph goes to $-\infty$ linearly as $D \to \infty$. 
\begin{theorem} \label{thm:lb}
Suppose $G$ is any graph with $D_w = \max_{x \sim y} \frac{\deg(x)}{w_{xy}}$ and $D_\mu = \max \frac{\deg(x)}{\mu(x)}$.  Then $G$ satisfies $CDE\left(2,-D_\mu\left(\frac{D_w}{2}+1\right)\right)$
\end{theorem}
\begin{proof}
Fix a function $f: V \to \R$ with $f > 0$, and vertex $x$ so that $\Delta f(x) < 0$.  We begin by calculating:
\begin{align}
\Gammatt(f)(x) =& \frac{1}{2} \left[ \Delta \Gamma(f) - 2 \Gamma\left(f, \frac{\Delta f^2}{2f} \right)\right] \nonumber \\
=& \frac{1}{2} \left[ \asum{y \sim x}(\Gamma(f)(y) - \Gamma(f)(x)) - \frac{1}{2} \asum{y \sim x} (f(y)-f(x)) \left( \frac{(\Delta f^2)(y)}{f(y)} - \frac{(\Delta f^2)(x)}{f(x)} \right)\right]\nonumber \\
=& 
\frac{1}{4} \asum{y \sim x} \asum{z \sim y} \left[ (f(z)-f(y))^2 - (f(y)-f(x))  \frac{(f^2(z)-f^2(y))}{f(y)} \right] \nonumber \\&\;\;\;\;\;\;\;\;\;\;- \frac{1}{2}\asum{y \sim x} \Gamma(f)(x) + \frac{1}{4}\asum{y \sim x} (f(y)-f(x))\frac{(\Delta f^2)(x)}{f(x)} \nonumber \\
&= \frac{1}{4} \asum{y \sim x}\asum{z \sim y} \left[ \frac{f(x)}{f(y)} f^2(z) - 2f(y)f(z) + 2f^2(y) - f(x)f(y) \right] \nonumber
\\&\;\;\;\;\;\;\;\;\;\;- \frac{1}{2}\asum{y \sim x} \Gamma(f)(x) + \frac{1}{2} \left( 
(\Delta f(x))^2 + \frac{\Gamma(f)}{f(x)} (\Delta f)\right),  \label{eq:startpt}
\end{align}
where in the second to last line we collected the terms at distance two, and in the
last line we used the identity that $(\Delta f^2)(x) = 2f(x)(\Delta f)(x) + 2\Gamma(f)(x)$.

The summands of the double sum are quadratics in $f(z)$. They are minimized when 
$f(z) = \frac{f^2(y)}{f(x)}$, whence the summand is $-\frac{f(y)}{f(x)}(f(x)-f(y))^2$, so
\begin{align}
\Gammatt(f) &\geq -\frac{1}{4} \asum{y \sim x} \asum{z \sim y} \frac{f(y)}{f(x)}(f(x)-f(y))^2 - \frac{1}{2}\asum{y \sim x} \Gamma(f)(x) + \frac{1}{2} \left( (\Delta f(x))^2 + \frac{\Gamma(f)}{f(x)} (\Delta f)\right)\nonumber\\
&\geq -\frac{1}{4}D_\mu \sum_{y \sim x} \frac{f(y)}{f(x)}(f(x)-f(y))^2 - \frac{1}{2}D_\mu \Gamma(f)(x)  + \frac{1}{2} \left( (\Delta f(x))^2 + \frac{\Gamma(f)}{f(x)} (\Delta f)\right). \label{eq:bd}
\end{align}
We use the fact that
\[
\Delta f = \asum{y \sim x}(f(y)-f(x)) \geq -\asum{y \sim x} f(x) \geq -D_{\mu}f(x), 
\]
to lower bound the $\frac{\Gamma(f)}{f(x)}(\Delta f)$ term.  Finally, we use the fact 
that $\Delta f < 0$, and Lemma \ref{prop:easy} (i) implies that 
\[
\frac{f(y)}{f(x)} < D_w. 
\]
Therefore, continuing from (\ref{eq:bd}), 
\begin{align*}
\Gammatt(f) &\geq  -\frac{1}{4}D_\mu \sum_{y \sim x} \frac{f(y)}{f(x)}(f(x)-f(y))^2 -
\frac{1}{2}D_\mu \Gamma(f)(x)  + \frac{1}{2} \left( (\Delta f(x))^2 + \frac{\Gamma(f)}{f(x)} (\Delta f)\right) \\
& > \frac{1}{2} (\Delta f(x))^2 - D_\mu \left(\frac{D_w}{2}+1\right) \Gamma(f) 
\end{align*}
as desired.

\end{proof}

\subsection{Sharpness of Theorem~\ref{thm:lb} on trees}
For unweighted graphs with the normalized Laplacian, Theorem \ref{thm:lb} 
states that all graphs satisfy $CDE(2,-\frac{D}{2} -1)$.  Such a lower bound on curvature is essentially tight in the case of trees.  
Indeed,
let $(T_D, x_0)$ denote the infinite $D$-ary tree rooted at $x_0$.  We find below functions $f_{D}$ for which
\begin{equation}
\frac{\Gammatt(f_D)}{\Gamma(f_D)} \leq -(1+o(1))\frac{D}{2},\label{eq:sharp} \mbox{ as } D\to\infty.
\end{equation}
To construct the function $f_D$ we do the following.  
Let $y_1, \dots, y_D$ denote the neighbors of $x_0$.  We define functions $f_\epsilon$ as follows:
\begin{align*}
f_{\epsilon}(x_0) &= 1\\
f_{\epsilon}(y_1) &= (1-\epsilon)D\\
f_{\epsilon}(y_i) &= \epsilon &&  \mbox{for $2 \leq i \leq D$.}
\end{align*}
For vertices $z \sim y_i$ at distance two from $x_0$, we take 
$f_\epsilon(z) = f^2(y_i)$ (and hence, by the computation in the proof of Theorem \ref{thm:lb} being the value that minimize $\Gammatt(f_\epsilon)$ given the $f_\epsilon(y_i)$).  
Then we take $f_{D} = f_\epsilon$ for $\epsilon = D^{-3/2}$. It is a straight forward
computation to verify that (\ref{eq:sharp}) holds.

\subsection{Ricci-flat graphs}

Chung and Yau \cite{ChY} introduced the notion of Ricci-flat (unweighted) graphs as a generalization of Abelian Cayley graphs.

\begin{definition} A $d$-regular graph $G(V,E)$ is Ricci-flat at the vertex $x\in V$ if there exists maps $\eta_i : V \to V; \; i = 1,\dots, d$ that satisfy the following conditions.
\begin{enumerate}
\item $x\eta_i(x) \in E$ for every $x\in V$.
%\item There is a fixed sequence of numbers $p_1, \dots, P_d$ such that $p_{v \eta_i(v)} = p_i$ depends only on $i$.
\item $\eta_i(x) \neq \eta_j(x)$ if $i \neq j$, for every $x \in V$.
\item for every $i$ we have $\cup_j \eta_i (\eta_j(x)) = \cup_j \eta_j (\eta_i(x))$
%\item 
%\item Whenever $\eta_j(\eta_i(x)) = \eta_i(\eta_k(x))$ then $p_j = p_k$. 
\end{enumerate}  

In fact to test Ricci-flatness at $x$ it is sufficient for the $\eta_i$s to be defined only on $x$ and the vertices adjacent to $x$. 

Finally, the graph $G$ is Ricci-flat if it is Ricci-flat at every vertex. 
\end{definition}

Given a weighted graph which is Ricci-flat when viewed as an unweighted graph, the
weighting is called {\it consistent} if 
\begin{enumerate}
\item There exist numbers $w_1, \dots, w_d$ so that $w_{x\eta_i(x)} = w_i$ for all $i= 1, \dots, d$ and  $x \in V$.  
\item Whenever $\eta_j(\eta_i(x)) = \eta_i(\eta_k(x))$ for some $x\in V$ then $w_j = w_k$. 
\item The weights are symmetric, so $w_{xy}=w_{yx}$ whenever $x \sim y$.  
\end{enumerate}
If only the first two conditions holds (so the weights are not necessarily symmetric) then 
we say the weighting is {\it weakly consistent.}

\begin{remark}
The conditions on the weights are fairly restrictive, but there are two cases when they are easily seen to be satisfied. 
\begin{enumerate}
\item If $w_i = 1 : i = 1,\dots, d$ then we get back the original notion of Ricci-flat graph. 
\item If $G$ is Ricci flat, and the functions $\eta_i$ locally commute, that is $\eta_i(\eta_j(x)) = \eta_j(\eta_i(x))$, then any sequence $w_1,\dots, w_d$ can be used to introduce a weakly consistent weighting for $G$. 
\end{enumerate}
\end{remark}

The critical reason why we choose these restrictions is the following:  If $G$ is a (weakly) consistently weighted  Ricci-flat graph and $f:V \to \mathbb{R}$ is a function, then for
any vertex $x \in V$, and $1 \leq i \leq d$, 
\begin{equation}
\sum_{j} w_j f(\eta_i\eta_j(x)) = \sum_j w_j f(\eta_j\eta_i(x)). \label{reason}
\end{equation}
Here the fact $G$ is Ricci flat implies the sums are over the same set of vertices, and the second condition on the weights ensures that the sums are equal.  

\begin{theorem}\label{thm:ricciflat} Let $G$ be a $d$-regular Ricci-flat graph. 
Suppose that the measure $\mu$ defining $\Delta$ satisfies $\mu(x) \equiv \mu$ for all
vertices $x \in G$. 
\begin{enumerate}
\item   If the weighting of $G$ is consistent, then $G$ satisfies $CDE(d,0)$.  
\item If the weighting of $G$ is weakly consistent, then $G$ satisfies $CDE(\infty,0)$
\end{enumerate}
\end{theorem}
\begin{remark}For a $d$-regular Ricci-flat graph and a weakly consistens weighting the two standard choices of the measure $\mu \equiv 1$ and $\mu(x)=\deg(x)$ satisfy $\mu(x)\equiv \mu$ for all $x\in V$.
\end{remark}
\begin{corollary}\label{corol:zd} The usual Cayley graph of $\mathbb{Z}^k$ satisfies $CDE(2k,0)$, for the regular or normalized graph Laplacian.  
\end{corollary}

\begin{proof}[Proof of Theorem~\ref{thm:ricciflat}] Let $f : V \to \R$ be a function. 

We begin by assuming that $G$ is Ricci flat, and the weighting is weakly consistent.  

We will write $y$ for $f(x)$, $y_i$ for $f(\eta_i(x))$, and $y_{ij}$ for $f(\eta_j(\eta_i(x)))$. With this notation we have
\begin{align*}
\Delta \Gamma(f)(x) &= \frac{1}{\mu} \sum_{i} w_i \left( \Gamma(f)(\eta_i(x)) - \Gamma(f)(x) \right)\\  &= \frac{1}{2\mu^2} \sum_{i} \sum_{j} w_iw_j \left((y_{ij}-y_i)^2 - (y_j -  y)^2 \right) \\&= \frac{1}{2\mu^2}\sum_{i,j} w_iw_j((y_{ij}^2 - y_j^2) + (y_i^2 - y^2) - 2y_iy_{ij} + 2yy_j), 
\end{align*} and 

\begin{align*}
2\Gamma\left(f, \frac{\Delta f^2}{2f}\right) &= \frac{1}{2\mu^2}  
\sum_i \sum_j w_iw_j( y_i - y)\left( \frac{y_{ij}^2 - y_i^2}{y_i} - \frac{y_{j}^2 - y^2}{y} \right) \\ 
&= \frac{1}{2\mu^2}  
\sum_i \sum_j w_iw_j( y_i - y)\left( \frac{y_{ji}^2 - y_i^2}{y_i} - \frac{y_{j}^2 - y^2}{y} \right)\\
&= \frac{1}{2\mu^2}  
\sum_i \sum_j w_iw_j( y_j - y)\left( \frac{y_{ij}^2 - y_j^2}{y_i} - \frac{y_{i}^2 - y^2}{y} \right)\\
&= \frac{1}{2\mu^2} \sum_i\sum_j w_iw_j \left( (y_{ij}^2 - y_j^2) + (y_i^2-y^2) + 2yy_j 
- \frac{y^2y_{ij}^2 + y_i^2y_j^2}{yy_j} \right).
\end{align*}
Here, the second equality follows from the (weakly) consistent labeling as observed in
(\ref{reason}) and the third equality follows from changing the role of $i$ and $j$.  

Combining, we see
\begin{align}
\Gammatt(f) = \frac{1}{2} \left(\Delta \Gamma(x) - 2\Gamma\left(f,\frac{\Delta f^2}{2f}\right)\right) &= \frac{1}{4\mu^2} \sum_{ij} w_iw_j \left( \frac{y^2y_{ij}^2 - 2y_iy_jy_{ij}y + y_i^2y_j^2}{yy_j} \right) \nonumber\\&= \frac{1}{4\mu^2} \sum_{ij} w_iw_j \frac{(yy_{ij} - y_iy_j)^2}{yy_j}. \label{gttbound}
\end{align}

Clearly, $\Gammatt(f) \geq 0$, so $G$ satisfies $CDE(\infty,0)$ proving the first 
part of the assertion.    

Now we further assume that the weighting of $G$ is consistent.  (That is, we further
assume the weights are symmetric.)  Now for each $i$ there is a unique $j = j(i)$ such that $\eta_j(\eta_i(x)) =x$ and thus $y_{ij} = y$. Throwing away all the other terms from (\ref{gttbound}) we get:
\[
\Gammatt \geq \frac{1}{4\mu^2} \sum_{i}  w_iw_{j(i)} \frac{(y^2 - y_iy_{j(i)})^2}{yy_i} .
\]

Note that $j(i)$ is a full permutation, and the symmetry of weights implies that $w_{i} =
w_{j(i)}$, and hence on the cycles in $j(i)$ the weights are constant.  Suppose the 
permutation $j(i)$ decomposes into cycles $C_1, \dots, C_k$, with lengths 
$\ell_1, \dots, \ell_k$.  We focus our attention on an arbitrary cycle $C$.  Then there
exists a $w_C$, and the terms above corresponding to this cycle are of the form

\[
\frac{w_C^2}{4\mu^2} \sum_{i \in C}  \frac{(y^2 - y_iy_{j(i)})^2}{yy_i}  = 
\frac{w_C^2y^2}{4\mu^2} \sum_{i \in C} \frac{(1 - z_iz_{j(i)})^2 }{z_i} = 
\frac{w_C^2 y^2}{4\mu^2} \sum_{i \in C} \left(\frac{1}{z_i} - 2z_{j(i)} + z_iz_{j(i)}^2 \right),
\]

where we take $z_i = y_i/y$.  
We can assume without loss of generality that $j(i)$ restricted to this cycle $C$ is a
permutation on $[\ell]$, and  $0 < z_1 \leq z_2 \leq \dots \leq z_\ell$. We can apply the Rearrangement Inequality to obtain $\sum z_i z_{j(i)}^2 \geq \sum z_i z_{\ell+1-i}^2$ and hence 
\begin{align}
\frac{w_C^2y^2}{4\mu^2}\sum_{i \in C} \left(\frac{1}{z_i} - 2z_{\ell+1-i} + z_i z_{\ell+1-i}^2\right) &= \frac{w_C^2y^2}{8\mu^2} \sum_{i \in C} (1-z_i z_{\ell+1-i})^2 \left( \frac{1}{z_i}+ \frac{1}{z_{\ell+1-i}}\right) \nonumber \\&\geq \frac{w_C^2y^2}{4\mu^2} \sum_{i \in C} \frac{(1-z_i z_{\ell+1-i})^2}{\sqrt{z_i z_{\ell+1-i}}}  \nonumber \\
& \geq \frac{w_C^{2}y^2}{\mu^2} \sum_{i \in C} (1 - \sqrt{z_{i}z_{\ell + 1 - i}})^2 \nonumber\\
&= \frac{1}{\mu^2}  \sum_{i \in C} (w_C (y - \sqrt{y_iy_{\ell+1-i}}))^2.
\label{eq:rearr}
\end{align}
We now combine the cycles together and apply Cauchy-Schwarz, to see 
\begin{equation}
\Gammatt(f) \geq \frac{1}{d} \left( \frac{1}{\mu}\sum_{i} w_i (y - \sqrt{y_iy_{i'}}) \right)^2, \label{eq:step}
\end{equation}
where $y_{i'}$ is the partner of $y_i$ in its cycle as given in (\ref{eq:rearr}).

Finally, we assume that $\Delta f(x) < 0$ to prove $CDE$.  This implies that $\sum_i w_i y_i < \sum_i w_i y$.  Also from the fact that $y_i$ and $y_i'$ appear in the same cycle, we have $\sum_i w_i y_i' = \sum_i w_i y_i$.  Applying Cauchy-Scwharz we see that
\[
\sum_i w_i \sqrt{y_iy_{i'}} \leq \sqrt{ \left(\sum_i w_i y_i\right) \left(\sum_i w_i y_i'\right)} = \sum_{i} w_iy_i  
< \sum_i w_i y.
\]
Thus continuing $(\ref{eq:step})$, we see the interior square is positive, and hence 
\[
\Gammatt(f) \geq \frac{1}{d} \left( \frac{1}{\mu}\sum_{i} w_i (y - \sqrt{y_iy_{i'}}) \right)^2  
\geq \frac{1}{d} \left ( \frac{1}{\mu} \sum_i w_i (y - y_i) \right)^2 = \frac{1}{d} (\Delta f)^2  
\]
as desired. 
\end{proof}

\subsection{Strong cut-off function in $\Z^d$}

\begin{prop}\label{prop:zdstrong} The usual Cayley graph of $\Z^d$, along with a strongly consistent weighting, admits a $(100,R)$-strong cut-off function supported in a ball of radius $\sqrt{d}R$ centered at the origin.
\end{prop}
\begin{remark}
In the case of the Cayley graph of $\Z^d$, a strongly consistent weighting just means that for each of the $d$ generators $e_i$, $w_{xe_i(x)} = w_{xe_{i}^{-1}(x)}$.  
\end{remark}
We did not attempt to optimize the constant 100 appearing in this statement.

\begin{proof}
For a vertex  $x \in \Z^d$ let $x_i \in \Z$ denote its $i$th coordinate and write $|x|^2 = \sum_i x_i^2$. We are going to prove that the function
\[ \phi(x) = \left(\max\left\{ 0,  \frac{R^2 - |x|^2}{R^2} \right\}\right)^2 \]
is a $(100,R)$-strong cut-off function centered at the origin. It is supported in a ``Euclidean'' ball of radius $R$ which is contained in a ball of radius $\sqrt{d}R$ measured in the graph distance.

We need to show that one of the two cases in Definition~\ref{def:strong_cutoff} are satisfies. If $R^2 - |x|^2 \leq 10R$ then 
the first case is clearly satisfied, so we may assume $R^2 - |x|^2 > 10R$. Also, $|x_i| < R$ for any $i$, otherwise $\phi(x)$ would be 0. These together imply that 
\begin{align} 
\frac{R^2 - |x|^2}{R^2 - |x|^2 \pm 2|x_i|+1} &\leq \frac{1}{1 - \frac{2|x_i|-1}{R^2-|x|^2}} \leq \frac{1}{1 - \frac{3R}{10R}} \leq \frac{10}{7}. 
\label{eq:bounding}
\end{align}
%R^2 - |x|^2 + 2|x_i| -1 &> 4R, \label{eq:biggerthan1}\\
%R^2 - |x|^2 - 2|x_i| -1 &> 2R  \mbox{, and}  \label{eq:biggerthan2}\\
%R^2 -|x|^2 - 2|x_i| - 1 &>  \frac{1}{2} R^2 - |x|^2. \label{eq:biggerthan3}
%\end{align}
By the consistency, for each coordinate there is a single weight $w_i$.  
Now we can compute
\begin{multline*} \mu(x)\phi^2(x) \Delta \frac{1}{\phi}(x) = \left(\frac{R^2-|x|^2}{R^2}\right)^4 \frac{R^4}{2}\cdot \\ \cdot \sum_{i} w_i\left(\frac{1}{(R^2-|x|^2-2|x_i|-1)^2} + \frac{1}{(R^2-|x|^2+2|x_i|-1)^2} - \frac{2}{(R^2-|x|^2)^2}\right)  \\=\left(\frac{R^2-|x|^2}{R^2}\right)^2\cdot\\
\cdot \sum_{i} w_i \left(\frac{(R^2-|x|^2)^2((R^2-|x|^2-1)^2+4x_i^2) - ((R^2-|x|^2-1)^2 - 4x_i^2)^2}{(R^2-|x|^2-2|x_i|-1)^2(R^2-|x|^2+2|x_i|-1)^2}\right) \\ \leq
\frac{1}{R^4} \sum_{i} w_i \left( \frac{12x_i^2(R^2-|x|^2)^4 + 2(R^2-|x|^2)^5}{(R^2-|x|^2-2|x_i|-1)^2(R^2-|x|^2+2|x_i|-1)^2} \right) .
\end{multline*}
In the last line, we used that $(R^2-|x|^2 - 1) \leq (R^2-|x|^2)$ and discarded some
negative terms.  Then using (\ref{eq:bounding}) along with $x_i^2 < R^2$ and $R^2 - |x|^2 < R^2$, we have 
\[
\phi^2(x) \Delta \frac{1}{\phi}(x) \leq \frac{1}{\mu(x)}\sum_i w_i \left(2 \cdot \frac{(10/7)^4}{R^2} \right) < \frac{100}{R^2}D_\mu. 
\]

A computation similar in spirit, but less complicated, shows that
\begin{multline*}\phi^3(x)\Gamma\left(\frac{1}{\phi}\right)(x) =\\ \frac{(R^2-|x|^2)^6}{2R^{12}\mu(x)} \sum_{i} w_i \left|\frac{R^4}{(R^2-|x|^2)^2} - \frac{R^4}{(R^2-|x|^2\pm 2|x_i| -1)^2}\right|^2 \leq \frac{100}{R^2}D_{\mu},
\end{multline*}
 and thus $\phi$ indeed is a $(100,R)$-strong cut-off function.
\end{proof}

\section{Applications} \label{app}

\subsection{Heat Kernel Estimates and Volume Growth}

One of the fundamental applications of the Li-Yau inequality, and more generally
parabolic Harnack inequalities, is the derivation of heat kernel estimates.  As alluded
to in the introduction, Grigor'yan and Saloff-Coste (in the manifold setting) and Delmotte (in the graph setting) proved the equivalence of several conditions (including Harnack inequalities, and the combination of volume doubling and the Poincar\'e inequality) to the heat kernel satisfying the following Gaussian type bounds.  Let $P_{t}(x,y)$ denote the fundamental solution to the heat equation 
starting at $x$.  

\begin{definition} $G$ satisfies the Gaussian heat-kernel property $\mathcal{G}(c,C)$ if $d(x,y)\leq t$ implies
\[
\frac{c}{\vol(B(x,\sqrt{t}))} \exp\left(-C \frac{d(x,y)^2}{t} \right) \leq P_t(x,y) 
\leq \frac{C}{\vol(B(x,\sqrt{t}))} \exp\left(-c\frac{d(x,y)^2}{t}\right).
\]
\end{definition}
In the graph setting, Delmotte proved that $\mathcal{G}(c,C)$ is equivalent to two
other (sets of) properties.  The first is the pair of volume doubling and Poincar\'e. 

\begin{definition} $G$ satisfies the volume doubling property $\mathcal{VD}(C)$ if 
for all $x \in V$ and all $r \in \mathbb{R}^+$:
\[
\vol(B(x,2r)) \leq C \vol(B(x,r))
\]
\end{definition}

\begin{definition} $G$ satisfies the Poincar\'e inequality $\mathcal{P}(C)$ if
\[
\sum_{x \in B(x_0,r)} \mu(x) (f(x) - f_B)^2 \leq C r^2 \sum_{x,y \in B(x_0,2r)} w_{xy}(f(y)-f(x))^2,
\]
for all $f: V \to \mathbb{R}$, for all $x_0 \in V$ and for all $r \in \R^+$, where
\[
f_B = \frac{1}{\vol(B(x_0,r))} \sum_{x \in B(x_0,r)} \mu(x)f(x).
\]
\end{definition}

The final equivalent condition is a Harnack inequality in the following form:

\begin{definition}  Fix $0 < \theta_1 
< \theta_2 < \theta_3 < \theta_4$ and $C > 0$.  $G$ satisfies the Harnack inequality property $\mathcal{H}(\theta_1,\theta_2,\theta_3,\theta_4,C)$ if for all $x_0 
\in V$ and $t_0, R \in \mathbb{R}^{+}$, and every positive solution $u(x,t)$ to the heat equation on $Q = B(x_0,2R) \times [s, s+\theta_4 R^2]$,
\[
\sup_{Q^-} u(x,t) \leq C \inf_{Q^+} u(x,t),
\]
where $Q^- = B(x_0,R) \times [s+\theta_1R^2, s+\theta_2R^2]$, and $Q^+ = B(x_0,R) \times [s+\theta_3 R^2,s+\theta_4R^2]$.  
\end{definition}

Delmotte shows that $\mathcal{H}(\theta_1,\theta_2,\theta_3,\theta_4,C_0) \Leftrightarrow 
\mathcal{P}(C_1) + \mathcal{VD}(C_2) \Leftrightarrow \mathcal{G}(c_3,C_4)$ for graphs,
the equivalent statement for manifolds is due to Grigor'yan and Saloff-Coste.  In 
the manifold case, it is well known that non-negative curvature implies $\mathcal{VD}$ and $\mathcal{P}$, but on graphs it is not known. Here, we show 
that $CDE(n,0)$ implies $\mathcal{H}$ (and hence all the properties) under the assumption that $G$ admits a $(c,\eta R)$ strong cutoff function contained in a ball $B(x_0,R)$ around every point.  For instance, the strong cutoff function for the
integer lattice $\Z^d$ shows we can guarantee a $(c, \frac{1}{\sqrt{d}}R)$ strong cutoff 
function in balls of radius $R$.  

\begin{corollary}[Corollary of Theorem \ref{thm:harnack}] Suppose $G$ satisfies $CDE(n,0)$, and let $\eta \in (0,1)$.  
If for every $x\in B(x_0, R)$ $G$ admits a $(c,\eta R)$-strong  cutoff function centered at $x$ with support in $B(x_0, 2R)$ then $G$ satisfies
$\mathcal{H}(\theta_1,\theta_2,\theta_3,\theta_4,C_0)$ for some $C_0$ (and therefore
$\mathcal{G}(c,C)$, $\mathcal{P}(C)$ and $\mathcal{VD}(C)$ for appropriate constants).
\end{corollary}
\begin{proof}
The proof is almost immediate from Theorem \ref{thm:harnack}.  Fix $\theta_1 < \theta_2 < \theta_3 < \theta_4$.  From Theorem 
\ref{thm:strongballgre} $G$ satisfies a gradient estimate of the form
\[
2(1-\alpha)\frac{\Gamma(\sqrt{u})}{u} - \frac{\Delta u}{u} \leq \frac{c_1}{t} + \frac{c_2}{R^2}
\]
on $B(x_0,R)$.  For $T_1 \in [s+\theta_1R^2, s+\theta_2R^2]$ and $T_2 \in 
[s + \theta_3R^2, s+\theta_4R^2]$, 
\[
\frac{T_2}{T_1} \leq \frac{s+\theta_4 R^2}{s+\theta_1R^2} \leq 1 + \frac{(\theta_4 - \theta_1)R^2}{s + \theta_4R^2} \leq 1 + \frac{\theta_4-\theta_1}{\theta_4}.
\]

Furthermore 
\[
\frac{c_2}{R^2} \cdot (T_2 - T_1) \leq c_2(\theta_4 - \theta_1)
\]
and
\[
\frac{d(x,y)^2}{T_2-T_1} \leq \frac{4}{\theta_3 - \theta_2}.
\]
Thus each of the terms arising in the Harnack inequality derived in Theorem \ref{thm:harnack} are bounded by constants not depending on $s$, $x_0$ and $R$, so we
can choose a $C_0$ guaranteeing that $\mathcal{H}(\theta_1,\theta_2,\theta_3,\theta_4,C_0)$ holds. 

\end{proof}

In general, however, we only have for graphs satisfying $CDE(n,0)$ the gradient 
estimate derived from Theorem \ref{thm:weakballgre}.  Using this gradient estimate
in Theorem \ref{thm:harnack} implies that
\[
u(x,T_1) \leq u(y,T_2) \cdot \left(\frac{T_2}{T_1}\right)^{c_1} \exp\left( \frac{c_2}{R}(T_2-T_1) + c_3\frac{d(x,y)^2}{T_2-T_1}  \right).
\]
This will not suffice for proving $\mathcal{H}(\theta_1,\theta_2,\theta_3,\theta_4,C_0)$.  Indeed, if $T_2 - T_1 = cR^{2}$, then this only implies that 
\[\sup_{Q^-} u(x,t) \leq \exp(cR+c') \inf_{Q^+} u(x,t),\] where the constant depends now on $R$. 

Nevertheless, we can derive heat kernel upper bounds that are Gaussian, and lower bounds that are not quite Gaussian but still have a similar form. The heat kernel bound then allows us to derive volume growth bounds:
we show that if $G$ satisfies $CDE(n,0)$ then $G$ has polynomial volume growth.  We derive here only on-diagonal upper and lower bounds, but it is known that off-diagonal 
bounds can be established using the on-diagonal bounds.  

\begin{theorem}
\label{thm:gaussian-bounds}
Suppose $G$ satisfies $CDE(n,0)$ and has maximum degree $D$.  Then there exist constants so that, for $t > 1$, 
\[
C \frac{1}{t^{n}} \exp\left(-C'\frac{d^2(x,y)}{t-1}\right)  \leq P_t(x,y) \leq C'' \frac{\mu(y)}{\vol(B(x,\sqrt{t}))}. 
\] \label{simplebounds}
\end{theorem}
\begin{proof}
The upper bound is standard and follows from the methods of Delmotte from \cite{Delmotte}.  
Indeed, the only observation is that the only time a Harnack inequality is
utilized in the proof of the upper bound, it is used on a solution to the heat equation
which is not just in the ball, but everywhere.  For such a function, letting $R \to \infty$ we observe that if $u$ is a solution on the whole graph, with $c_1 = n$, then
\begin{equation}
u(x,T_1) \leq u(y,T_2) \left(\frac{T_2}{T_1}\right)^n \exp\left(\frac{d(x,y)^2D}{(1-\alpha)(T_2-T_1)}\right).  \label{eq:harnform}
\end{equation}
Then the argument proceeds as follows.  Let $P_\cdot(\cdot, y)$ be the fundamental solution to the heat equation.  Then by (\ref{eq:harnform}), for $u=P_t$ if $z \in B(x,\sqrt{t})$,
\[
P_t(x,y) \leq P_{2t}(z,y) 2^{n} \exp\left(\frac{D}{1-\alpha} \right) = C' \cdot P_{2t}(z,y)
\]
Thus
\begin{align*}
P_t(x,y) &\leq \frac{C}{\vol(B(x,\sqrt{t}))} \sum_{z \in B(x,\sqrt{t})} \mu(z) P_{2t}(z,y).   \\
& \leq \frac{C}{\vol(B(x,\sqrt{t}))} \sum_{z \in B(x,\sqrt{t})} \mu(y) P_{2t}(y,z). \\
& \leq \frac{C' \mu(y)}{\vol(B(x,\sqrt{t}))}.  
\end{align*}
This gives the desired upper bound.  

The lower bound proceeds directly from the Harnack inequality (\ref{eq:harnform}).

Indeed,
\[
P_1(y,y) \leq P_{t}(x,y) t^{n} \exp\left(C'd(x,y)^2/(t-1)\right).
\]
Noting that $P_1(y,y)$ is bounded from below by an absolute constant in a bounded degree
graph and dividing yields the result.  
\end{proof}
An immediate corollary of Theorem \ref{thm:gaussian-bounds} is polynomial volume growth. 
\begin{corollary}
\label{cor:polynomial-growth}
Let $G$ be a graph satisfying $CDE(n,0)$. Then $G$ has a
polynomial volume growth.
\end{corollary}
\begin{proof}
Applying Theorem~\ref{thm:gaussian-bounds} with $y=x$ gives
\[
\frac{C}{t^n} \leq \frac{C' \mu(x)}{\vol(B(x,\sqrt{t}))},
\]
and cross multiplying yields the desired bounds.  
\end{proof}

\subsection{Buser's inequality for graphs} 

As another application of the gradient estimate in Theorem \ref{thm:compactgrecurv} we prove a Buser-type~\cite{Buser} estimate for the smallest nontrivial eigenvalue of a finite graph. For now on we assume that the edge weights are symmetric, i.e. $w_{xy}=w_{yx}$ for all $x\sim y$. 

In the following we denote 
$$\|f\|_p = \left(\sum_{x\in V}\mu(x)f^p(x)\right)^{\frac{1}{p}} \text{ and } \|f\|_\infty = \sup_{x\in V}|f(x)|.$$
The Cheeger constant $h$ of a graph is defined as
\[h = \inf_{\emptyset\neq U\subset V: \mathrm{vol}(U)\leq 1/2\, \mathrm{vol}(V)}\frac{|\partial U|}{\mathrm{vol}(U)},\]
where $|\partial U| = \sum_{x\in U, y\in V\setminus U}w_{xy}$ and $\mathrm{vol}(U) = \sum_{x\in U}\mu(x)$.
\begin{theorem}\label{thm:buser}
Let $G$ be a finite graph satisfying $CDE(n,-K)$ for some $K > 0$ and  fix $0 < \alpha < 1$. Then
\[\lambda_1 \leq \max\{2C \sqrt{K}h, 4C^2 h^2\},\] where the constant
\[C=8 \left(3\mu_{\mathrm{max}}\frac{(2-\alpha)n}{\alpha(1-\alpha)^2}\right)^{\frac{1}{2}}\]
 only depends on the dimension $n$ and $\mu_{\mathrm{max}}$.
\end{theorem}
\begin{remark} 
{\hspace{0.1in}}
\begin{itemize}
\item By using Theorem \ref{thm:compactgre} instead of Theorem \ref{thm:compactgrecurv}, one obtains the same statement in the case of $K=0$ where now the constant $C$ is given by $C=8\sqrt{3n\mu_{\mathrm{max}}}.$
\item The Cheeger inequality states that $\frac{h^2}{2D_\mu}\leq\lambda_1$. Thus in particular if $K= 0$,  Theorem \ref{thm:buser} implies that  $\frac{h^2}{2D_\mu}\leq \lambda_1 \leq 4 C^2 h^2$ , i.e. $\lambda_1$ is of the order $h^2$.
\item Klartag and Kozma~\cite{KK} show a similar but stronger result for graphs satisfying the original CD-inequality. Namely they prove, following the arguments of Ledoux \cite{Ledoux2}, that if a finite graphs satisfies  $CD(\infty, -K)$ then 
\[ \lambda_1 \leq 8\max\{ \sqrt{K} h, h^2\}. \]  
Note that their condition does not involve dimension, and hence their constant is also dimension independent.
\end{itemize}
\end{remark}
We divide the proof into several different steps, closely following Ledoux's \cite{Ledoux} argument on compact manifolds. The proof of the following lemma is based on ideas by Varopoulos~\cite{Var}.
\begin{lemma}\label{lem:infnormest} Let
$G$ be a finite graph satisfying $CDE(n,-K)$ for some $K > 0$, and let $P_tf$ be a positive solution to the heat equation on $G$. Fix $0 < \alpha < 1$ and let
$0<t\leq t_0$ then
\[\|\Gamma(P_tf)\|_\infty \leq \frac{12 c}{(1-\alpha)t}\|f\|_\infty^2, \]
where $c =  \frac{n}{2(1-\alpha)} + \frac{Kn}{\alpha}t_0$.

\end{lemma}
\begin{proof}
On the one hand by the gradient estimate Theorem  \ref{thm:compactgrecurv}  and $t \leq t_0$
\[ \frac{(1-\alpha)\Gamma(\sqrt{P_tf})}{P_tf} - \frac{\Delta P_tf}{2P_tf} \leq  \frac{n}{2(1-\alpha)t} +  \frac{Kn}{\alpha}\frac{t_0}{t} =:\frac{c}{t}.\]
Since $ \frac{(1-\alpha)\Gamma(\sqrt{P_tf})}{P_tf}\geq 0$ and the estimate is trivial if $\frac{\Delta P_tf}{2P_tf}\geq 0$ we conclude that
\begin{equation}\label{eq:buser1}\left(\frac{\Delta P_tf}{2P_tf}\right)^-\leq \frac{c}{t},\end{equation} where $(\quad )^\pm$ denotes the positive and negative part, respectively.
Note that $0= \sum_{x\in V}\mu(x)\Delta P_tf(x) =  \sum_{x\in V}\mu(x)\left(\Delta P_tf\right)^+(x) - \mu(x)\left(\Delta P_tf\right)^-(x)$ which implies
\begin{equation}\label{eq:buser2}\sum_{x\in V}\mu(x)\left(\Delta P_tf\right)^-(x) = \frac{1}{2}\sum_{x\in V}\mu(x)(\left(\Delta P_tf\right)^-(x) +\left(\Delta P_tf\right)^+(x))= \frac{1}{2} \|\Delta P_tf\|_1.\end{equation}
Moreover since $\sum_{x\in V}\mu(x) P_tf(x) = \sum_{x\in V}\mu(x) f(x)$ and $f>0$ it follows from \eqref{eq:buser1} and \eqref{eq:buser2} that
\begin{equation}\label{eq:1norm}\frac{1}{4} \|\Delta P_tf\|_1  = \frac{1}{2}\sum_{x\in V} \mu(x)\left(\Delta P_t f\right)^- \leq \frac{c}{t}\sum_{x\in V}\mu(x)P_tf(x) = \frac{c}{t}\|f\|_1.\end{equation}
It is well know that for bounded linear operators $T: \ell^p \to \ell^q$  and their dual operators $T^*:\ell^{q^*}\to\ell^{p^*}$  it holds that 
$$\|T\|_{ \ell^p \to \ell^q} =\|T^*\|_{\ell^{q^*}\to\ell^{p^*}}$$ where 
$$\|T\|_{A\to B} := \sup_{f\in A}\frac{\|Tf\|_B}{\|f\|_A}$$ and $p$ and $p^*$ are H\"older conjugate exponents, i.e. $\frac{1}{p} + \frac{1}{p^*}=1$. Sine $\Delta P_t$ is self-adjoint we have for all $f$
$$\frac{\|\Delta P_tf\|_\infty}{\|f\|_\infty}\leq \|\Delta P_t\|_{\infty\to \infty}= \|\Delta P_t\|_{1\to 1} = \sup_{g\in\ell^1}\frac{\|\Delta P_tg\|_1}{\|g\|_1} \leq \frac{4c}{t}.$$
On the other hand it follows from the gradient estimate by applying the infinity norm on both sides that
\begin{eqnarray}\label{eq: infnormest}
(1-\alpha)\|\Gamma(\sqrt{P_tf})\|_\infty&\leq & \nonumber \frac{1}{2}\|\Delta P_tf\|_\infty + \frac{c}{t}\|P_t f\|_\infty\\&\leq&
 \frac{2c}{t}\|f\|_\infty +  \frac{c}{t}\|f\|_\infty =  \frac{3c}{t}\|f\|_\infty
\end{eqnarray} where we used \eqref{eq:1norm} and $\|P_tf\|_\infty \leq \|P_0f\|_\infty = \|f\|_\infty$ for all $t>0$.
Now the proof is almost complete, we only need to estimate $\Gamma(\sqrt{P_tf})$ by $\Gamma(P_tf)$. It is easy to see that
$\Gamma(u)\leq 4 \|u\|_\infty\Gamma(\sqrt{u})$ for all positive functions $u>0$. Indeed,
\begin{eqnarray*}\Gamma(u)(x) &=&  \frac{1}{2\mu(x)}\sum_{y\sim x}w_{xy}\left(u(x) -u(y)\right)^2 \\ &=& \frac{1}{2\mu(x)}\sum_{y\sim x}w_{xy}\left(\sqrt{u(x)} -\sqrt{u(y)}\right)^2 \left(\sqrt{u(x)} +\sqrt{u(y)}\right)^2\\& \leq&   4 \|u\|_\infty\Gamma(\sqrt{u}).
\end{eqnarray*}
Using this in \eqref{eq: infnormest} we obtain
\[\|\Gamma(P_tf)\|_\infty \leq \frac{12 c}{(1-\alpha)t}\|f\|_\infty^2, \]
 which finishes the proof.
\end{proof}
\begin{remark}
Using the notation  $|\nabla f| = \sqrt{\Gamma(f)}$ the statement of the last lemma is equivalent to
\begin{equation}\label{eq: infnormest2} \| |\nabla P_tf|\|_\infty \leq 2\sqrt{\frac{3 c}{(1-\alpha)t}}\|f\|_\infty.\end{equation}
\end{remark}
\begin{lemma}\label{lem:onenorm}
Let $G$ be a finite graph satisfying $CDE(n,-K)$ for some $K > 0$, and let $P_tf$ be a positive solution to the heat equation on $G$. Fix $0 < \alpha < 1$ and let $0<t\leq t_0$ then
\[\|f-P_tf\|_1\leq 8 \sqrt{\frac{3c}{1-\alpha}}\||\nabla f| \|_1 \sqrt{t},\] where $c$ is the constant in Lemma \ref{lem:infnormest}.
\end{lemma}
\begin{proof}
For any positive function $g$ we have
\begin{multline*}
\sum_{x\in V} \mu(x) g(x)(f - P_tf)(x)  =  \sum_{x\in V} \mu(x) g(x)(P_0f - P_tf)(x) = \\ -\int_0^t \sum_{x\in V} \mu(x) g(x)\frac{\partial}{\partial s}P_sf(x)ds  =    -\int_0^t
\sum_{x\in V} \mu(x) g(x)\Delta P_sf(x) ds= \\ -\int_0^t \sum_{x\in V} \mu(x) P_sg(x)\Delta f(x)ds  = \int_0^t \sum_{x\in V} \mu(x) \Gamma(P_s g, f)(x)ds,
\end{multline*} where we used that $P_s = e^{s\Delta}$ is self-adjoint, $P_s$ commutes with $\Delta$, and summation by parts.
Applying Cauchy-Schwarz  and H\"older we obtain
\begin{eqnarray*}
\sum_{x\in V} \mu(x) g(x)(f - P_tf)(x) & \leq&   \int_0^t \sum_{x\in V} \mu(x) |\nabla P_s g|(x)  |\nabla f|(x)ds\\&\leq &
 \int_0^t \||\nabla P_s g|\|_\infty  \||\nabla f|\|_1ds.
\end{eqnarray*}
Applying \eqref{eq: infnormest2} yields
\begin{multline}\label{eq: infnormest3}
\sum_{x\in V} \mu(x) g(x)(f - P_tf)(x) \leq \\  \int_0^t \sqrt{\frac{12 c}{1-\alpha}}\frac{1}{\sqrt{s}}\|g\|_\infty  \||\nabla f|\|_1ds\leq 4 \sqrt{\frac{3c}{1-\alpha}}\|g\|_\infty \||\nabla f| \|_1 \sqrt{t}.\end{multline}
Now assume for the moment that $\sum_{x\in V} \mu(x)(f-P_tf)(x)\geq 0$. 
We choose $g = \mathrm{sgn}(f - P_t f) + 1+\epsilon$ for some $\epsilon >0$ such that $g$ is positive and 
\begin{eqnarray*}
\|f-P_tf\|_1 &\leq& \sum_{x\in V}\mu(x)|f-P_tf|(x) +(1+\epsilon) \sum_{x\in V} \mu(x)(f-P_tf)(x) 
\\ &= &\sum_{x\in V} \mu(x) g(x)(f - P_tf)(x)\leq 
(1+\epsilon)8 \sqrt{\frac{3c}{1-\alpha}} \||\nabla f| \|_1 \sqrt{t}
 \end{eqnarray*}
 where we used \eqref{eq: infnormest3} and $\|g\|_\infty=2$. Taking $\epsilon \to 0$ completes the proof. 
If $\sum_{x\in V} \mu(x)\cdot(f-P_tf)(x)< 0$ then we choose $g = \mathrm{sgn}(P_t f-f) + 1+\epsilon$ and the proof is completed in the same way as above. 
\end{proof}
With these preparations we can now prove Theorem \ref{thm:buser}.
\begin{proof}[Proof of Theorem \ref{thm:buser}]
We want to apply Lemma \ref{lem:onenorm} to the characteristic function $\chi_U$ of any subset $U$.
The left hand side becomes
\begin{eqnarray}\label{eq:buser3}\nonumber 8 \sqrt{\frac{3c}{1-\alpha}}\||\nabla \chi_U| \|_1 \sqrt{t} &=& 8 \sqrt{\frac{3c}{1-\alpha}}\sqrt{t}\sum_{x\in V}\mu(x) \sqrt{\frac{1}{2\mu(x)}\sum_{y\sim x}w_{xy}(\chi_U(y)-\chi_U(x))^2}\\&\leq& \nonumber
 8 \sqrt{\frac{3c}{1-\alpha}}\sqrt{t}\sum_{x\in V}\sqrt{\frac{\mu(x)}{2}} \sum_{y\sim x}w_{xy}|\chi_U(y)-\chi_U(x)|
\\&\leq&  8 \sqrt{\frac{3c}{1-\alpha}}\sqrt{t}\sqrt{2\mu_{\mathrm{max}}}|\partial U|\end{eqnarray}  where $\mu_{\mathrm{max}}=\max_{x\in V}\mu(x)$.

The right hand side becomes:
\begin{eqnarray*}
\|\chi_U - P_t\chi_U\|_1 &=& \sum_{x\in U}\mu(x)|\chi_U(x) - P_t\chi_U(x)| +  \sum_{x\in V\setminus  U}\mu(x)|\chi_U(x) - P_t\chi_U(x)|
\\&=&   \sum_{x\in U}\mu(x)(1- P_t\chi_U(x)) +  \sum_{x\in V\setminus  U}\mu(x) P_t\chi_U(x)\\&=& 2( \mathrm{vol}(U) - \sum_{x\in U}\mu(x) P_t\chi_U(x))\\&=&2( \|\chi_U\|_2^2 - \|P_{t/2}\chi_U\|_2^2)
\end{eqnarray*} where we used that $P_{\frac{t}{2}}P_{\frac{t}{2}} =P_t$, $P_t\chi_U \leq 1$,  $\mathrm{vol(U)} =  \sum_{x\in U}\mu(x) P_t\chi_U(x)+\\ $ \mbox{$+\sum_{x\in V\setminus U} \mu(x)P_t\chi_U(x)$} and the fact that $P_t$ is self-adjoint.
Let $\{\psi_i\}_{i=0}^{N-1}$ ($N$ is the number of vertices in the graph) be an orthonormal basis of eigenfunctions, i.e. 
$$(\psi_i,\psi_j) = \sum_{x\in V}\mu(x) \psi_i(x)\psi_j(x) = \delta_{ij}.$$ In particular the eigenfunction
corresponding to the trivial eigenvalue $\lambda_0 =0$ is given by $\psi_0 = \frac{1}{\sqrt{\mathrm{vol}(V)}}$. Then every function $f:V\to \mathbb{R}$  can be
expanded in the basis $\{\psi_i\}$, i.e. $f = \sum_{i=0}^{N-1}\alpha_i\psi_i$, where $\alpha_i =(f,\psi_i)= \sum_{x\in V}\mu(x) f(x)\psi_i(x)$. For the
characteristic function this gives $\chi_U =  \sum_{i =0}^{N-1}\alpha_i\psi_i$ with $\alpha_0 = \sum_{x\in V}\mu(x) \chi_U \frac{1}{\sqrt{\mathrm{vol}(V)}}=
\frac{\mathrm{vol}(U)}{\sqrt{\mathrm{vol}(V)}}$. Since the $\psi_i$ form an orthonormal basis we have
\[\|\chi_U\|_2^2 = \sum_{x\in V}\mu(x) \sum_{i=0}^{N-1}\alpha_i^2\psi_i^2(x)=  \sum_{i=0}^{N-1}\alpha_i^2= \mathrm{vol}(U).\] By the spectral theorem,
\[P_{t}(\chi_U) = \sum_{i=0}^{N-1}e^{-\lambda_i t}\alpha_i\psi_i \] and thus
\[ \|P_{t/2}\chi_U\|_2^2=  \sum_{i=0}^{N-1} e^{-\lambda_i t}\alpha_i^2\leq  e^{-\lambda_1 t}\sum_{i=1}^{N-1}\alpha_i^2 + \alpha_0^2.\]Combining everything we obtain
\begin{equation}\label{eq:buser4}2( \|\chi_U\|_2^2 - \|P_{t/2}\chi_U\|_2^2 \geq 2(1-e^{-\lambda_1t})\sum_{i=1}^{N-1}\alpha_i^2 = 2(1-e^{-\lambda_1t})\left(\mathrm{vol(U) -\frac{\mathrm{vol}(U)^2}{\mathrm{vol}(V)}}\right). \end{equation}
From now on we choose $t_0 = K^{-1}$. The reason is that for this particular choice the constant $c$ is independent of the curvature bound $K$. From \eqref{eq:buser3} and \eqref{eq:buser4}  we have for all $0< t\leq K^{-1}$ and all subsets $U$ of $V$ for which $\mathrm{vol}(U)\leq \frac{1}{2}\mathrm{vol}(V)$
\[\frac{|\partial(U)|}{\mathrm{vol(U)}}\geq \frac{ (1-e^{-\lambda_1t})}{C\sqrt{t}},\] where\[C= 8 \left(\frac{6c \mu_{\mathrm{max}}}{1-\alpha}\right)^{\frac{1}{2}}.\]
 Since this is true for every subset $U\subset V$ and $0<t<K^{-1}$ this implies
\[h\geq \frac{1}{C}\sup_{0<t\leq K^{-1}} \frac{ (1-e^{-\lambda_1t})}{\sqrt{t}}.\]
Now if $\lambda_1\geq K$, we choose $t = \frac{1}{\lambda_1}$ which yields
\[h \geq \frac{1}{C}(1 - \frac{1}{e})\sqrt{\lambda_1}\geq \frac{1}{2C}\sqrt{\lambda_1},\]
while if $\lambda_1\leq K$ we take $t = K^{-1}$ which yields
\[h \geq \frac{1}{C}\sqrt{K}(1- e^{-\frac{\lambda_1}{K}})\geq \frac{1}{2C\sqrt{K}}\lambda_1.\] This yields
\[\lambda_1 \leq \max\{2C \sqrt{K}h, 4C^2 h^2\}\] which completes the proof.
\end{proof}
%\begin{acknowledgements}
\noindent\textbf{Acknowledgements.}
We are very grateful to Bo'az Klartag and Gady Kozma for showing us their unpublished manuscript~\cite{KK} and introducing us to the work of Ledoux.
Bauer was partially supported by the Alexander von Humboldt foundation and partially supported by the NSF Grant DMS-0804454 Differential Equations in Geometry.
Horn, Lippner and Yau were supported by AFOSR Grant FA9550-09-1-0090.
Lin was partially supported National Natural Science Foundation of China, Grant No.~11271011.
Lippner, Mangoubi and Yau would like to acknowledge  the support of BSF grant no.~2010214.
%\end{acknowledgements}

\noindent
Frank Bauer, Paul Horn, Gabor Lippner, Shing-Tung Yau,\\
Department of Mathematics, Harvard University, Cambridge, Massachusetts \\ 
\texttt{[fbauer, phorn, lippner, yau]@math.harvard.edu} 

\smallskip
\noindent
Yong Lin,\\ 
Department of Mathematics, Renmin University of China, Beijing, China \\
\texttt{linyong01@ruc.edu.cn} 

\smallskip
\noindent
Dan Mangoubi,\\ 
Department of Mathematics, Hebrew University, Jerusalem, Israel\\
\texttt{mangoubi@math.huji.ac.il}

\end{document}